\documentclass[11pt,reqno]{amsart}

\usepackage[T1]{fontenc}
\usepackage{tgpagella}
\usepackage{mathpazo}

\usepackage{hyperref} 
\usepackage{xcolor}
\usepackage{amsmath,amsthm}
\usepackage{amssymb}
\usepackage{graphicx}
\usepackage[mathscr]{eucal}
\usepackage{MnSymbol}
\usepackage[frame,ps,matrix,arrow,curve,rotate,all,2cell,tips]{xy}
\usepackage{epic,eepic}
\usepackage{enumerate}
\usepackage{color,soul}

\newtheorem{theorem}[subsection]{Theorem}
\newtheorem{lemma}[subsection]{Lemma}

\theoremstyle{definition}
\newtheorem{definition}[subsection]{Definition}

\newtheorem{remark}[subsection]{Remark}
\newtheorem{assumption}[subsection]{Assumption}

\newtheorem{notation}[subsection]{Notation}

\numberwithin{equation}{subsection}

\newtheorem{problem}[subsection]{Problem}

\newcommand{\ac}{\scriptstyle \textrm{!`}}

\newcommand{\N}{\mathcal{N}}
\newcommand{\M}{\mathcal{M}}

\newcommand{\C}{\mathcal{C}}

\newcommand{\Ho}{\mathsf{Ho}}

\newcommand{\Comon}{\mathsf{Comon}}

\DeclareMathOperator{\Cobar}{Cobar}
\DeclareMathOperator{\sCoCoalg}{sCoCoalg}
\DeclareMathOperator{\dgAlg}{dgAlg}

\DeclareMathOperator{\Spec}{Spec}

\newcommand{\po}{\ar@{}[dr]|(.7){\Searrow}}
\newcommand{\pb}{\ar@{}[dr]|(.3){\Nwarrow}}

\DeclareMathOperator{\map}{\mathsf{map}}
\newcommand{\mapm}{\map_{\M}}
\newcommand{\mapcoalgkm}{\map_{\coalgkm}}

\newcommand{\Ch}{\mathsf{Ch}}
\newcommand{\Chr}{\Ch_R}
\newcommand{\ChA}{\Ch_{\cat A}}
\newcommand{\Chk}{\Ch_k}
\newcommand{\Chrinj}{\Ch_{R,\mathsf{inj}}}
\newcommand{\Comonchrinj}{\Comon\left(\Chrinj\right)}
\newcommand{\Chrb}{\Chr^B}
\newcommand{\Chrinjb}{\Chrinj^B}

\newcommand{\lcchrinj}{L_{\C}\Chrinj}

\newcommand{\coalgpchr}{\coalg\left(P; \Chr\right)}
\newcommand{\coalgqchr}{\coalg\left(Q; \Chr\right)}
\newcommand{\coalgqchrinj}{\coalg\left(Q; \Chrinj\right)}
\newcommand{\coalgqlcchrinj}{\coalg\left(Q; L_{\C}\Chrinj\right)}

\newcommand{\cat}[1]{\mathcal{#1}}

\newcommand{\Sp}{\mathsf{Sp}^{\Sigma}}
\newcommand{\Spinj}{\Sp_{\mathsf{inj}}}
\newcommand{\Sppr}{\Sp_{\mathsf{pr}}}
\newcommand{\SpprE}{\Sp_{\mathsf{pr},E}}

\newcommand{\lcmoda}{L_{\C}\Moda}
\newcommand{\lcmodab}{\left(\lcmoda\right)^B}

\newcommand{\adjoint}{
\nicearrow\xymatrix{ \ar@<2pt>[r] & \ar@<2pt>[l]}}
\renewcommand{\hookrightarrow}{\nicexy{\ar@{^{(}->}[r] &}}
\newcommand{\nicearrow}{\SelectTips{cm}{10}}
\newcommand{\nicexy}{\nicearrow\xymatrix@C+5pt}

\newcommand{\pushout}{\ar@{}[dr]|(0.75){\Searrow}}
\newcommand{\drrpushout}{\ar@{}[drr]|(0.90){\Searrow}}

\renewcommand{\to}{\hspace{-.1cm}\nicearrow\xymatrix@C-.2cm{\ar[r]&}\hspace{-.1cm}}

\newcommand{\sO}{\mathsf{O}}

\newcommand{\sS}{\mathsf{S}}

\newcommand{\xtilde}{\widetilde{X}}

\newcommand{\andspace}{\qquad\text{and}\qquad}

\newcommand{\forallspace}{\quad\text{for all}\quad}
\newcommand{\ifspace}{\quad\text{if}\quad}

\renewcommand{\lim}{\mathsf{lim}\,}

\DeclareMathOperator{\CoEnd}{CoEnd}
\DeclareMathOperator{\sMod}{sMod}
\DeclareMathOperator{\sPre}{sPre}
\DeclareMathOperator{\Hom}{Hom}

\DeclareMathOperator{\Id}{Id}

\newcommand{\Comod}{\mathsf{Comod}}
\newcommand{\Mod}{\mathsf{Mod}}
\newcommand{\Moda}{\Mod_A}
\newcommand{\Modab}{\Moda^B}

\newcommand{\tensorunit}{\mathtt{1}}

\newcommand{\lcm}{L_\C\M}
\newcommand{\LC}{L_{\C}}
\newcommand{\Lcprime}{L_{\C'}}

\newcommand{\RC}{R_{\C}}

\newcommand{\coalg}{\mathsf{Coalg}}
\newcommand{\coalgk}{\coalg(K)}
\newcommand{\coalgkm}{\coalg(K;\M)}

\newcommand{\coalgklcm}{\coalg(K;\lcm)}

\newcommand{\RKC}{R^{K}_{\C}}

\newcommand{\QK}{Q^{K}}

\newcommand{\Uinv}{U^{-1}}

\newcommand{\Cof}{\mathsf{Cof}}

\newcommand{\We}{\mathsf{WE}}

\newcommand{\Cofm}{\Cof_{\M}}

\newcommand{\Wem}{\We_{\M}}

\newcommand{\Coflcm}{\Cof_{\lcm}}

\newcommand{\Welcm}{\We_{\lcm}}

\begin{document}

\title{Comonadic Coalgebras and Bousfield Localization}

\author{David White}
\address{Department of Mathematics \\ Denison University
\\ Granville, OH}
\email{david.white@denison.edu}

\author{Donald Yau}
\address{Department of Mathematics \\ The Ohio State University at Newark}
\email{yau.22@osu.edu}

\begin{abstract}
For a model category, we prove that taking the category of coalgebras over a comonad commutes with left Bousfield localization in a suitable sense. Then we prove a general existence result for the left-induced model structure on the category of coalgebras over a comonad in a left Bousfield localization. Next we provide several equivalent characterizations of when a left Bousfield localization preserves coalgebras over a comonad. These results are illustrated with many applications in chain complexes, (localized) spectra, the stable module category, and simplicial settings. 
\end{abstract}

\maketitle

\tableofcontents

\section{Introduction}

This paper sits at the intersection of two important threads in homotopy theory: Bousfield localization and comonad.  For a model category $\M$ and a class of morphisms $\C$ in $\M$, the left Bousfield localization $\lcm$, if it exists, is the model structure on $\M$ with the same cofibrations and with the morphisms in $\C$ turned into weak equivalences.  Bousfield localization is a ubiquitous tool in homotopy theory, going back at least as far back as \cite{bousfield} and continuing with the work of Hopkins, Ravenel \cite{ravenel}, and many others.  On the other hand, many interesting objects in homotopy theory arise as coalgebras over comonads, such as comonoids and comodules over a comonoid.

Suppose the left Bousfield localization $\lcm$ exists, and $K$ is a comonad on $\M$ whose category of coalgebras $\coalgkm$ admits a left-induced model structure over $\M$. This means a morphism of $K$-coalgebras is a weak equivalence (resp., cofibration) if and only if its underlying morphism in $\M$ is so. The main aim of this paper is to answer the following inter-related questions, first posed in \cite{white-thesis} (Problem 9.1.10).
\begin{description}
\item[(A)] If the category of $K$-coalgebras $\coalgklcm$ admits a left-induced model structure over $\lcm$, is it equal to a left Bousfield localization of $\coalgkm$?
\item[(B)] When does the category of $K$-coalgebras $\coalgklcm$ admit a left-induced model structure over $\lcm$?
\item[(C)] When does the left Bousfield localization $\LC$ preserve $K$-coalgebras?
\end{description}
The following result, which will be proved as Theorem \ref{clockwise.ccw}, provides a positive answer to question (A).

\begin{theorem}
Suppose $\coalgkm$ admits a left-induced model structure over $\M$. Then the following two statements are equivalent.
\begin{enumerate}
\item $\coalgklcm$ admits the left-induced model structure via the forgetful functor $U : \coalgklcm \to \lcm$.
\item The left Bousfield localization $\Lcprime \coalgkm$ exists, where $\C'$ is defined in Assumption \ref{standing.assumptions}.
\end{enumerate}
Furthermore, if either statement is true, then there is an equality
\[\coalgklcm = \Lcprime \coalgkm\]
of model categories.
\end{theorem}

An analogous result for a monad instead of a comonad is \cite{batanin-white} (3.5). For monads, we must use the right-induced model structure on algebras over a monad, where a morphism $f$ of algebras is a weak equivalence or fibration if and only if $U(f)$ is in $\M$. For a monad and \emph{right} Bousfield localization, an analogous result is \cite{white-yau4} (2.6). However, in practice the above result is applied differently from the analogous results in \cite{batanin-white,white-yau4}. The reason is that the right-induced model structures on categories of monadic algebras and the left-induced model structures on categories of comonadic coalgebras exist under very different circumstances. The following result, which will appear as Theorem \ref{coalgklcm-admissible}, provides an answer to question (B). Conditions under which $\coalgkm$ is locally presentable may be found in \cite{ching-riehl} (Prop. A.1).

\begin{theorem}
Suppose $\M$ is a combinatorial model category and $\lcm$ is a cofibrantly generated model category.  Suppose $\coalgkm$ is a locally presentable category that admits a left-induced model structure over $\M$. Then the category $\coalgklcm$ admits the left-induced model structure over $\lcm$.
\end{theorem}

Next we consider question (C) above. Monad-algebra structures are not in general preserved by localizations (\cite{white-localization} provides counterexamples). However, preservation of monad-algebra structures can happen under suitable conditions \cite{gutierrez-white-equivariant, hovey-white, white-localization,white-yau,white-yau3,white-yau6}. The following preservation result, which will appear as Theorem \ref{comonad.preservation.eq}, provides several equivalent characterizations of when a left Bousfield localization preserves coalgebras over a comonad.

\begin{theorem}
Suppose $\coalgkm$ (resp., $\coalgklcm$) admits the left-induced model structure over $\M$ (resp., $\lcm$).  Then the following statements are equivalent.
\begin{enumerate}
\item The forgetful functor $U : \coalgklcm \to \lcm$ preserves weak equivalences and fibrant objects.
\item $\LC$ preserves $K$-coalgebras (Def. \ref{def:preserves-coalgebra}).
\item $\LC$ lifts to the homotopy category of $K$-coalgebras (Def. \ref{crt-comonad}).
\item The forgetful functor preserves left Bousfield localization (Def. \ref{grso-comonad}).
\end{enumerate}
\end{theorem}

An analogous result for a monad instead of a comonad is \cite{batanin-white} (5.6), using the right-induced model structure on algebras. An analogue for a monad and right Bousfield localization is \cite{white-yau4} (5.4). Once again, in practice the implementation of the previous result is quite different from those in \cite{batanin-white,white-yau4}.  

The above theorems are proved in Sections \ref{sec:lifting-bousfield}--\ref{sec:equivalent-approach}.  The second half of this paper contains many applications of these theorems to various categories and left Bousfield localizations of interest.  In Section \ref{sec:chain-applications} we prove preservation results under homological truncations for comonoids and (coring) comodules in chain complexes over a commutative unital ring equipped with the injective model structure.  In Section \ref{sec:spectra-applications} we prove preservation results under smashing localizations for comodules over a comonoid in (localized) spectra with the injective model structure.  In Section \ref{sec:application-stmod} we prove preservation results under smashing localizations for comodules, comonoids, and coalgebras over suitable cooperads in the stable module category. In Section \ref{sec:simplicial-applications}, we provide preservation results in simplicial model categories, simplicial coalgebras, and simplicial presheaves.

\section*{Acknowledgments} The first author is grateful to Emily Riehl, Kathryn Hess, Brooke Shipley, Niles Johnson, Michael Ching, Andrew Salch, Bob Bruner, Pavel Safronov, and Paul Goerss for helpful conversations regarding the applications of our main theorems. We also thank Carles Casacuberta and the anonymous referee for many helpful comments.

\section{Lifting Bousfield Localization to Comonadic Algebras}
\label{sec:lifting-bousfield}

We assume the reader is familiar with the basics of model categories, as explained in \cite{hirschhorn,hovey}.

Suppose $K$ is a comonad on a model category $\M$ that admits a left Bousfield localization $\lcm$ with respect to some class of cofibrations $\C$ such that the category of $K$-coalgebras in $\M$ admits a left-induced model structure via the forgetful functor to $\M$.  The main result of this section (Theorem \ref{clockwise.ccw}) says that the category of $K$-coalgebras in $\lcm$ admits a left-induced model structure via the forgetful functor to $\lcm$ if and only if the category of $K$-coalgebras in $\M$ admits a suitable left Bousfield localization.  Furthermore, when either condition holds, the two model categories coincide.

We begin by recalling some definitions regarding left Bousfield localization and left-induced model structures.  The homotopy function complex in a model category $\M$ is denoted by $\mapm$  \cite{hirschhorn} (Def. 17.4.1).

\begin{definition}\label{def:lcm}
Suppose $\M$ is a model category, and $\C$ is a class of morphisms in $\M$.
\begin{enumerate}
\item A \emph{$\C$-local object} is a fibrant object $X \in \M$ such that the induced morphism
\[\nicexy@C+1cm{\mapm(A,X) & \mapm(B,X) \ar[l]_-{\mapm(f,X)}}\]
of simplicial sets is a weak equivalence for all the morphisms $f : A \to B$ in $\C$.
\item A \emph{$\C$-local equivalence} is a morphism $f: A \to B$ in $\M$ such that the induced morphism
\[\nicexy@C+1cm{\mapm(A,X) & \mapm(B,X) \ar[l]_-{\mapm(f,X)}}\]
of simplicial sets is a weak equivalence for all $\C$-local objects $X$.
\item Define a new category $\lcm$ as being the same as $\M$ as a category, together with the following distinguished classes of morphisms. A morphism $f$ in $\lcm$ is a:
\begin{itemize}
\item \emph{cofibration} if it is a cofibration in $\M$.
\item \emph{weak equivalence} if it is a $\C$-local equivalence.
\item \emph{fibration} if it has the right lifting property with respect to trivial cofibrations, i.e., morphisms that are both cofibrations and weak equivalences.
\end{itemize}
\item If $\lcm$ is a model category with these weak equivalences, cofibrations, and fibrations, then it is called the \emph{left Bousfield localization of $\M$ with respect to $\C$} \cite{hirschhorn} (Def. 3.3.1(1)), and the $\C$-local objects are precisely the fibrant objects in $\lcm$ \cite{hirschhorn} (Prop. 3.4.1). This model structure exists when $\C$ is a set of morphisms and $\M$ is left proper and either cellular or combinatorial, by \cite{hirschhorn} (Thm. 4.1.1(2)) or \cite{barwick} (Thm. 4.7).
\end{enumerate}
\end{definition}

\begin{remark} \label{remark:cof-between-cof-obj}
In the previous definition, using functorial cofibrant replacement and functorial factorization of morphisms into cofibrations followed by trivial fibrations, without loss of generality we may assume that morphisms in $\C$ are cofibrations between cofibrant objects. Also, because $\lcm$ has the same cofibrations as $\M$, if $\M$ has all objects cofibrant then so does $\lcm$. Furthermore, 
if $\M$ is left proper, simplicial, combinatorial, or cellular, then $\lcm$ is the same \cite{lurie} (Prop.\ A.3.7.3).
\end{remark}

When an adjunction is drawn, the left adjoint will always be drawn on top.

\begin{definition}[\cite{bayeh,gkr,hkrs}]
Suppose $L : \N \adjoint \M : R$ is an adjunction with left adjoint $L$ and $\M$ is a model category.  We say that $\N$ admits the \emph{left-induced model structure} via $L$ if it admits the model category structure in which a morphism $f$ is a weak equivalence (resp., cofibration) if and only if $Lf$ is a weak equivalence (resp., cofibration) in $\M$.
\end{definition}

\begin{remark}
In the previous definition, when $\N$ admits the left-induced model structure via the left adjoint $L$, the adjoint pair $(L,R)$ is a Quillen adjunction because $L$ preserves cofibrations and trivial cofibrations.
\end{remark}

\begin{definition}
\begin{enumerate}
\item A \emph{comonad} $(K,\delta,\eta)$ on a category $\M$ \cite{maclane} (VI.1) consists of a functor $K : \M \to \M$ and natural transformations $\delta : K \to K^2$ and $\eta : K \to \Id$ such that the following coassociativity and counity diagrams commute:
\[\nicexy{K \ar[d]_-{\delta} \ar[r]^-{\delta} & K^2 \ar[d]^-{K\delta}\\
K^2 \ar[r]^-{\delta K} & K^3}\qquad 
\nicexy{K & K^2 \ar[l]_-{\eta K} \ar[r]^-{K\eta} & K\\
& K\ar[ul]^-{=} \ar[u]_-{\delta} \ar[ur]_-{=}}\]
\item A \emph{$K$-coalgebra} in $\M$ is a pair $(X,\lambda)$ consisting of an object $X \in \M$ and a morphism $\lambda : X \to KX$ such that the following coassociativity and counity diagrams commute:
\[\nicexy{X \ar[d]_-{\lambda} \ar[r]^-{\lambda} & KX \ar[d]^-{K\lambda}\\
KX \ar[r]^-{\delta_X} & K^2X}\qquad
\nicexy{X \ar[dr]_-{=} \ar[r]^-{\lambda} & KX \ar[d]^-{\eta_X}\\ & X}\]
A morphism of $K$-coalgebras $(X,\lambda) \to (Y,\rho)$ is a morphism $f : X \to Y$ in $\M$ compatible with the structure morphisms $\lambda$ and $\rho$ in the sense that the diagram
\[\nicexy{X \ar[d]_-{\lambda} \ar[r]^-{f} & Y \ar[d]^-{\rho}\\ KX \ar[r]^-{Kf} & KY}\]
is commutative.
\item With $(K,\delta,\eta)$ abbreviated to $K$, the category of $K$-coalgebras is denoted by $\coalgkm$.  The corresponding forgetful-cofree adjunction is denoted by
\begin{equation}\label{forgetful-cofree}
\nicexy{\coalgkm \ar@<2pt>[r]^-{U} & \M \ar@<2pt>[l]^-{K}}.
\end{equation}
In this adjunction, the forgetful functor $U$ is the left adjoint, and the cofree $K$-coalgebra functor is the right adjoint.
\end{enumerate}
\end{definition}

\begin{assumption}\label{standing.assumptions}
Suppose that:
\begin{enumerate}
\item $\M$ is a model category, and $\C$ is a class of cofibrations between cofibrant objects in $\M$ such that the left Bousfield localization $\lcm$ exists.  
\item $K$ is a comonad on $\M$ such that the category $\coalgkm$ admits the left-induced model structure via the forgetful functor $U$ to $\M$. 
\item $\C'$ is the class of morphisms $f : A \to B$ in $\coalgkm$  such that the induced morphism
\begin{equation}\label{cprime.def}
\nicexy@C+2cm{\mapcoalgkm(A,KX) & \mapcoalgkm(B,KX) \ar[l]_-{\mapcoalgkm(f,KX)}}
\end{equation}
of simplicial sets is a weak equivalence for all $\C$-local objects $X$ in $\M$, where $KX$ is the cofree $K$-coalgebra of $X$ \eqref{forgetful-cofree}.
\end{enumerate}
\end{assumption}

Many examples of the left-induced model structure on $\coalgkm$ are given in \cite{bayeh}, \cite{wald}, and \cite{hkrs}, some of which will be used below.

\begin{lemma}\label{cprime-local-eq}
Under Assumption \ref{standing.assumptions}, the following statements are equivalent for a morphism $f$ in $\coalgkm$.
\begin{enumerate}
\item $f \in \C'$.
\item $f$ is a $\C'$-local equivalence.
\item $Uf$ in $\M$ is a $\C$-local equivalence.
\end{enumerate}
\end{lemma}

\begin{proof}
First we show that (1) and (2) are equivalent.  By the definition of $\C'$-local objects, every morphism in $\C'$ is a $\C'$-local equivalence.  Conversely, note that for each $\C$-local object $X$, the cofree $K$-coalgebra $KX$ is a $\C'$-local object.  Indeed, since the forgetful-cofree adjunction \eqref{forgetful-cofree} is a Quillen adjunction and since $X \in \M$ is fibrant, the object $KX \in \coalgkm$ is fibrant. So $KX$ is a $\C'$-local object by the definition of the class $\C'$.  It follows that every $\C'$-local equivalence $f$ induces a weak equivalence of simplicial sets in \eqref{cprime.def}, and therefore $f$ belongs to $\C'$.  

Next we show that (1) and (3) are equivalent.  Suppose $f^{\bullet} : A^{\bullet} \to B^{\bullet}$ is a cosimplicial resolution of $f$ in $\coalgkm$ \cite{hirschhorn} (Def. 16.1.2(1) and Prop. 16.1.22(1)).  Applying the forgetful functor $U$ entrywise to $f^{\bullet}$ yields a cosimplicial resolution of $Uf$ in $\M$ because cofibrations and weak equivalences in $\coalgkm$ are defined in $\M$ via $U$.

For any $\C$-local object $X$ in $\M$, as noted above, the cofree $K$-coalgebra $KX$ is a fibrant object in $\coalgkm$.  Using the forgetful-cofree adjunction \eqref{forgetful-cofree}, there is a commutative diagram of simplicial sets:
\[\nicexy@C+2cm{\mapm(UA,X)  \ar@{=}[d] & \mapm(UB,X) \ar[l]_-{\mapm(Uf,X)} \ar@{=}[d]\\
\M\bigl(UA^{\bullet}, X\bigr) \ar[d]_-{\cong} & \M\bigl(UB^{\bullet},X\bigr) \ar[l]_-{\M(Uf^{\bullet},X)} \ar[d]_-{\cong}\\
\coalgkm\bigl(A^{\bullet}, KX\bigr) \ar@{=}[d] & \coalgkm\bigl(B^{\bullet},KX\bigr) \ar[l]_-{\coalgkm(f^{\bullet},KX)} \ar@{=}[d]\\
\mapcoalgkm(A,KX) & \mapcoalgkm(B,KX) \ar[l]_-{\mapcoalgkm(f,KX)}}\]
By definition $f \in \C'$ if and only if the bottom horizontal morphism in the above diagram is a weak equivalence for all $\C$-local objects $X$.  By commutativity and the 2-out-of-3 property, this is equivalent to the top horizontal morphism being a weak equivalence for all $\C$-local objects $X$.  This in turn is equivalent to $Uf$ in $\M$ being a $\C$-local equivalence.
\end{proof}

\begin{theorem}\label{clockwise.ccw}
Under Assumption \ref{standing.assumptions}, the following two statements are equivalent.
\begin{enumerate}
\item $\coalgklcm$ admits the left-induced model structure via the forgetful functor $U : \coalgklcm \to \lcm$.
\item The left Bousfield localization $\Lcprime \coalgkm$ exists.
\end{enumerate}
Furthermore, if either statement is true, then there is an equality
\[\coalgklcm = \Lcprime \coalgkm\]
of model categories.
\end{theorem}

This theorem says that, in the diagram
\[\nicexy@C+1cm{
\M \ar@{|->}[d]_-{\text{left-induced}} \ar@{|->}[r]^-{\LC} & \lcm \ar@{|.>}[r]^-{\text{left-induced}}_-{\text{exists ?}} & \coalgklcm \ar@{=}[d]^-{?}\\
\coalgkm \ar@{|.>}[rr]^-{\text{$\Lcprime$ exists ?}} && \Lcprime\coalgkm}\]
the ability to go counter-clockwise is equivalent to the ability to go clockwise.  Furthermore, when either one is possible, the results are equal.

\begin{proof}
The categories $\coalgklcm$ and $\Lcprime\coalgkm$ are both equal to the category $\coalgkm$.  To show the model structures coincide, it suffices to show that they have the same classes of cofibrations and also the same classes of weak equivalences.  In each of these two categories, a cofibration is a morphism $f : A \to B$ in $\coalgkm$ such that $Uf$ in $\M$ is a cofibration.  Moreover, the equivalence of (2) and (3) in Lemma \ref{cprime-local-eq} says that these two categories have the same weak equivalences.
\end{proof}

\section{Admissibility of Comonadic Coalgebras in Bousfield Localization}
\label{sec:admissibility-coalgebra}

In order to apply Theorem \ref{clockwise.ccw}, we will need to know when condition (1) or condition (2) there holds.  In the main result of this section (Theorem \ref{coalgklcm-admissible}), we prove that under mild conditions the category $\coalgklcm$ admits a left-induced model structure via the forgetful functor to the left Bousfield localization $\lcm$.

\begin{notation}
Suppose $\M$ is a category, $\C$ is a class of morphisms in $\M$, and $f : A \to B$ and $g : C \to D$ are morphisms in $\M$.
\begin{enumerate}
\item We say that $f$ has the \emph{left lifting property} with respect to $g$ if every solid-arrow commutative diagram
\[\nicexy{A \ar[d]_-{f} \ar[r] & C \ar[d]^-{g}\\ B \ar[r] \ar@{.>}[ur] & D}\]
in $\M$ admits a dotted arrow, called a \emph{lift}, that makes the entire diagram commutative.  In this case, we also say that $g$ has the \emph{right lifting property} with respect to $f$ and write $f \boxslash g$.
\item Define $\C^{\boxslash}$ to be the class of morphisms $h$ in $\M$ such that $c \boxslash h$ for all $c$ in $\C$.
\end{enumerate}
\end{notation}

Next we provide reasonable conditions under which the two equivalent statements in Theorem \ref{clockwise.ccw} are true.  We will make use of the following right-to-left transfer principle from \cite{bayeh} (Theorem 2.23), \cite{gkr} (Corollary 2.7), and \cite{hkrs} (Corollary 3.3.4(2)).  Note that for a locally presentable category, the weaker notion of \emph{cofibrantly generated} in \cite{bayeh} (Def. 2.4) for a pair of sets of generating morphisms coincides with the usual one \cite{hovey} (Def. 2.1.17).  Recall that a \textit{combinatorial model category} is a cofibrantly generated model category whose underlying category is locally presentable \cite{adamek-rosicky} (Def. 1.17). The following theorem is from \cite{bayeh,gkr,hkrs}.

\begin{theorem}\label{bayeh-hkrs}
Suppose $U : \N \adjoint \M : F$ is an adjoint pair with left adjoint $U$, $\N$ a locally presentable category, and $\M$ a combinatorial model category.  Then the following two statements are equivalent.
\begin{enumerate}
\item There is an inclusion
\begin{equation}\label{acyclicity}
\left(\Uinv \Cofm \right)^{\boxslash} \subseteq \Uinv \Wem
\end{equation}
called the acyclicity condition, where $\Cofm$ and $\Wem$ are the classes of cofibrations and of weak equivalences in $\M$.
\item $\N$ admits the left-induced model structure via $U$.
\end{enumerate}
\end{theorem}

The following observation says that under mild conditions, if $\coalgkm$ admits the left-induced model structure over $\M$, then $\coalgklcm$ admits the left-induced model structure over $\lcm$. The third condition below is not overly restrictive, as can be seen in \cite{ching-riehl} (Prop. A.1). An analogue of (1) below was proven in \cite{wald} (Prop. A.6). Thanks to Remark \ref{remark:cof-between-cof-obj} the second assumption is no obstacle.

\begin{theorem}\label{coalgklcm-admissible}
Suppose:
\begin{itemize}
\item $\M$ is a combinatorial model category.
\item $\C$ is a class of cofibrations between cofibrant objects in $\M$ such that $\lcm$ is a cofibrantly generated model category.
\item $K$ is a comonad on $\M$ such that $\coalgkm$ is a locally presentable category that admits the left-induced model structure via the forgetful functor $U$ \eqref{forgetful-cofree}.
\end{itemize}
Then the following two statements hold.
\begin{enumerate}
\item The category $\coalgklcm$ admits the left-induced model structure via the forgetful functor $U$ in the forgetful-cofree adjunction
\begin{equation}\label{forgetful-cofree-lcm}
\nicexy{\coalgklcm \ar@<2pt>[r]^-{U} & \lcm \ar@<2pt>[l]^-{K}}.
\end{equation}
\item With $\C'$ as in Assumption \ref{standing.assumptions}, the left Bousfield localization $\Lcprime \coalgkm$ exists and is equal to the model category $\coalgklcm$.
\end{enumerate}
\end{theorem}

\begin{proof}
For the first assertion, by Theorem \ref{bayeh-hkrs} with $\N = \coalgkm$, the acyclicity condition \eqref{acyclicity} holds for the forgetful-cofree adjunction \eqref{forgetful-cofree}.  As categories both
\[\lcm = \M \andspace \coalgklcm = \coalgkm\]
are locally presentable.  Since $\lcm$ is a cofibrantly generated model category by assumption, it is a combinatorial model category.  It remains to check the acyclicity condition \eqref{acyclicity} for the forgetful-cofree adjunction \eqref{forgetful-cofree-lcm}.  We have that:
\[\begin{split}
\left(\Uinv \Coflcm\right)^{\boxslash}
&= \left(\Uinv \Cofm\right)^{\boxslash}\\
&\subseteq  \Uinv \Wem\\
&\subseteq  \Uinv \Welcm.
\end{split}\]
The first equality follows from $\Cofm = \Coflcm$.  The first inclusion is the acyclicity condition \eqref{acyclicity}, which we assumed is true.  The last inclusion follows from the inclusion
\[\Wem \subseteq \Welcm.\]
This proves the first assertion.  The second assertion follows from the first assertion and  Theorem \ref{clockwise.ccw}.
\end{proof}

\section{Preservation of Comonadic Coalgebras Under Bousfield Localization}
\label{sec:preservation-coalgebra}

Left Bousfield localization does not preserve algebraic structure in general.  For example, in the category of symmetric spectra with the stable model structure, the $(-1)$-Postnikov section is a left Bousfield localization that does not preserve monoids \cite{cgmv}.  Therefore, we should not expect left Bousfield localization to preserve comonadic coalgebras in general.  The main result of this section  (Theorem \ref{left-comonad-preservation}) provides conditions under which comonadic coalgebra structures are preserved by left Bousfield localization.  This preservation result is the comonad analogue of \cite{white-yau} (Theorem 7.2.3).  Its analogue for a monad and right Bousfield localization is \cite{white-yau2} (Theorem 6.2).  A strengthened version of this preservation result is Theorem \ref{comonad.preservation.eq} below.

\begin{definition}\label{def:preserves-coalgebra}
Under Assumption \ref{standing.assumptions}, we say that \emph{$\LC$ preserves $K$-coalgebras} if the following statements hold.
\begin{enumerate}
\item For each $K$-coalgebra $X$, there exists a $K$-coalgebra $\xtilde$ such that $U\xtilde$ is a $\C$-local object that is weakly equivalent to the localization $\LC UX$ in $\M$.
\item 
If $X$ is a cofibrant $K$-coalgebra, then:
\begin{enumerate}
\item There is a natural choice of $\xtilde$ as part of a $K$-coalgebra morphism $r : X \to \xtilde$. 
\item There is a weak equivalence $\beta : U\xtilde \to \LC UX$ in $\M$ such that
\[\beta \circ Ur = l\]
in $\Ho(\M)$, where $l : UX \to \LC UX$ is the localization morphism.
\end{enumerate}
\end{enumerate}
\end{definition}

Note the connection between the hypotheses of the next result and Theorems \ref{clockwise.ccw} and \ref{coalgklcm-admissible}.

\begin{theorem}\label{left-comonad-preservation}
Under Assumption \ref{standing.assumptions}, suppose $\coalgklcm$ admits the left-induced model structure via the forgetful functor  
\[\nicexy{\coalgklcm \ar[r]^-{U} & \lcm}\]
in the forgetful-cofree adjunction \eqref{forgetful-cofree-lcm}.  If this forgetful functor $U$ preserves fibrant objects, then $\LC$ preserves $K$-coalgebras.
\end{theorem}

\begin{proof}
Write:
\begin{itemize}
\item
$Q$ and $\QK$ for the cofibrant replacements in $\M$ and $\coalgkm$, respectively; 
\item
$\RC$ and $\RKC$ for the fibrant replacements in $\lcm$ and $\coalgklcm$, respectively.
\end{itemize}
Pick a $K$-coalgebra $X$.  The localization $\LC UX$ is weakly equivalent in $\M$ to $\RC QUX$ and that $\RKC \QK X$ comes with a $K$-coalgebra structure.  We will take $\RKC \QK X$ to be our $\xtilde$, so we must show that
\begin{equation}
\label{comonad-left-weq}
\RC Q UX \simeq U \RKC \QK X
\end{equation}
in $\M$.  The proof proceeds in three steps.

For the first step, the cofibrant replacements $Q$ and $\QK$ give the commutative diagram
\begin{equation}\label{comonad-left-step-1}
\nicexy{\varnothing \ar@{>->}[d] \ar@{>->}[r] & QUX \ar@{>>}[d]^-{\sim}\\
U \QK X \ar@{.>}[ur]^-{\alpha} \ar[r]^-{\sim} & UX}
\end{equation}
in $\M$, in which $\varnothing$ is the initial object in $\M$. The right vertical morphism is a trivial fibration in $\M$.  The forgetful functor
\[U : \coalgkm \to \M,\]
being a left adjoint, preserves colimits and, in particular, the initial object.   The morphism $\varnothing \to \QK X$ is a cofibration in $\coalgkm$.  After applying $U$ its underlying morphism is a cofibration in $\M$ because the model structure on $\coalgkm$ is left-induced from that of $\M$.  So the dotted filler $\alpha$ exists in $\M$.  Moreover, the morphism $Q^K X \to X$ is a trivial fibration in $\coalgkm$, so after applying $U$ the bottom horizontal morphism is a weak equivalence in $\M$.  The $2$-out-of-$3$ property now implies that $\alpha$ is a weak equivalence in $\M$.

For the second step, first recall that
\begin{equation}\label{same-categories}
\M = \lcm \andspace \coalgkm = \coalgklcm
\end{equation}
as categories.   The fibrant replacements $\RC$ and $\RKC$ give the commutative diagram
\begin{equation}\label{comonad-left-step-2}
\nicexy{U \QK X \ar@{>->}[d]^-{\sim}_-{Ur} \ar@{>->}[r]_-{\sim}^-{l} & \RC U \QK X \ar@{>>}[d]\\
U \RKC \QK X \ar@{.>}[ur]_-{\beta} \ar[r] & \ast}
\end{equation}
in $\lcm$, in which $\ast$ is the terminal object in $\M$.  The top-right composite is the fibrant replacement of $U \QK X$ in $\lcm$.  The left vertical morphism is $U$ applied to the fibrant replacement in $\coalgklcm$:
\[\nicexy{\QK X \ar@{>->}[r]_-{\sim}^-{r} & \RKC \QK X.}\]
In the diagram \eqref{comonad-left-step-2}, the right vertical morphism is a fibration in $\lcm$.  The morphism $r$ is a trivial cofibration in $\coalgklcm$, so after applying $U$ the left vertical morphism in \eqref{comonad-left-step-2} is a trivial cofibration in $\lcm$.  Therefore, the dotted filler $\beta$ exists in $\lcm$.  Furthermore, the top horizontal morphism in \eqref{comonad-left-step-2} is a weak equivalence in $\lcm$.  So the $2$-out-of-$3$ property implies that $\beta$ is a weak equivalence in $\lcm$.

For the last step, consider the morphisms
\begin{equation}\label{comonad-left-step-3}
\nicexy{U \RKC \QK X \ar[r]^-{\beta}
& \RC U \QK X \ar[r]^-{\RC\alpha} & \RC QUX}
\end{equation}
in $\lcm$, in which $\alpha$ and $\beta$ are the morphisms in \eqref{comonad-left-step-1} and \eqref{comonad-left-step-2}, respectively.  Since $\alpha$ is a weak equivalence in $\M$, the second morphism $\RC \alpha$ is a weak equivalence in $\lcm$ between $\C$-local objects.  Furthermore, it was established in the previous paragraph that the first morphism $\beta$ is a weak equivalence in $\lcm$.  Since $\RKC \QK X$ is a fibrant object in $\coalgklcm$, by assumption $U \RKC \QK X$ is a fibrant object in $\lcm$ (i.e., a $\C$-local object).  So \cite{hirschhorn} (Theorem 3.2.13(1)) implies that both morphisms in \eqref{comonad-left-step-3} are weak equivalences in $\M$.  Therefore, $\RKC \QK X$ is a $K$-coalgebra whose underlying object, namely $U \RKC \QK X$, is weakly equivalent to the localization
\[\LC UX \simeq \RC QUX\]
in $\M$.

Next suppose $X$ is a cofibrant $K$-coalgebra in $\coalgkm$; i.e., $UX$ is cofibrant in $\M$.  In this case, the localization $\LC UX$ is weakly equivalent in $\M$ to the fibrant replacement $\RC UX$ in $\lcm$, and we may simply take $\alpha$ to be the identity morphism on $UX$ in step 1 above.  What we proved above now says that
\[\nicexy{U \RKC X \ar[r]^-{\beta} & \RC UX}\]
is a weak equivalence in $\M$.  So the fibrant replacement in $\coalgklcm$,
\[\nicexy{X \ar[r]^-{r} & \RKC X,}\]
is the desired lift of the localization morphism $l : UX \to \LC UX$.
\end{proof}

\section{Equivalent Approaches to Preservation of Comonadic Coalgebras}
\label{sec:equivalent-approach}

The main result of this section (Theorem \ref{comonad.preservation.eq}) provides equivalent  characterizations of preservation of comonadic coalgebras under left Bousfield localization (Def. \ref{def:preserves-coalgebra}).  Simultaneously, we provide a converse to the preservation Theorem \ref{left-comonad-preservation}.

Aside from Def. \ref{def:preserves-coalgebra}, another approach to preservation of comonadic coalgebras under left Bousfield localization is based on the following definition, which is the comonad version of \cite{crt} (Def. 7.3).

\begin{definition}\label{crt-comonad}
Under Assumption \ref{standing.assumptions}, we say that \emph{$\LC$ lifts to the homotopy category of $K$-coalgebras} if the following statements hold.
\begin{enumerate}
\item There exists a natural transformation $r : \Id \to L^K$ for functors on $\Ho\bigl(\coalgkm\bigr)$.
\item There exists a natural isomorphism $h : \LC U \to UL^K$ such that
\begin{equation}\label{hlu=ur}
h \circ lU = Ur
\end{equation}
in $\Ho(\M)$, where $l : \Id \to \LC$ is the unit of the derived adjunction
\[\Ho(\M) \adjoint \Ho(\lcm)\]
and $U$ is the forgetful functor.
\end{enumerate}
\end{definition}

A third approach to preservation of comonadic coalgebras under left Bousfield localization is based on the following definition, which is the comonad version of \cite{grso} (3.12).

\begin{definition}\label{grso-comonad}
Under Assumption \ref{standing.assumptions}, suppose $\Lcprime \coalgkm$, the left Bousfield localization with respect to $\C'$, exists.  We say that the forgetful functor
\[\nicexy{\Lcprime\coalgkm \ar[r]^-{U} & \lcm}\]
\emph{preserves left Bousfield localization} if, given any morphism $c: X \to \Lcprime X$ that is a $\C'$-local equivalence with $\C'$-local codomain in $\coalgkm$, the morphism $Uc$ is a $\C$-local equivalence with $\C$-local codomain in $\M$.
\end{definition}

In the previous definition, by Theorem \ref{clockwise.ccw}, the category $\coalgklcm$ admits the left-induced model structure via the forgetful functor to $\lcm$, and it is equal to $\Lcprime\coalgkm$ as a model category.

The following result shows that the three approaches to preservation of comonadic coalgebras under left Bousfield localization are equivalent and, furthermore, provides a converse to Theorem \ref{left-comonad-preservation}. It is essentially the comonad analogue of \cite{batanin-white} (Thm. 5.6).  Moreover, its analogue for a monad and right Bousfield localization is \cite{white-yau4} (Thm. 5.4).

\begin{theorem}\label{comonad.preservation.eq}
Under Assumption \ref{standing.assumptions}, suppose $\coalgklcm$ admits the left-induced model structure via the forgetful functor  
\begin{equation}\label{coalgklcm.forget}
\nicexy{\coalgklcm = \Lcprime\coalgkm \ar[r]^-{U} & \lcm}
\end{equation}
in which the equality is from Theorem \ref{clockwise.ccw}.  Then the following statements are equivalent.
\begin{enumerate}
\item $U$ in \eqref{coalgklcm.forget} preserves weak equivalences and fibrant objects.
\item $\LC$ preserves $K$-coalgebras (Def. \ref{def:preserves-coalgebra}).
\item $\LC$ lifts to the homotopy category of $K$-coalgebras (Def. \ref{crt-comonad}).
\item The forgetful functor preserves left Bousfield localization (Def. \ref{grso-comonad}).
\end{enumerate}
\end{theorem}

\begin{proof}
(1) $\Longrightarrow$ (2) is Theorem \ref{left-comonad-preservation}.  For (2) $\Longrightarrow$ (3) we take the augmented endofunctor $r : \Id \to L^K$ in Def. \ref{crt-comonad}(1) to be the image of $r : X \to \xtilde$ (Def. \ref{def:preserves-coalgebra}(2a)) in $\Ho\bigl(\coalgkm\bigr)$.  Then we take the natural isomorphism
\[h^{-1} : \LC U \cong UL^K\]
in Def. \ref{crt-comonad}(2) to be the image of the morphism $\beta$ (Def. \ref{def:preserves-coalgebra}(2b)) in $\Ho(\M)$ .

For (3) $\Longrightarrow$ (1) first note that the forgetful functor $U$ preserves weak equivalences because the model structure on $\coalgklcm$ is left-induced from that on $\lcm$ via $U$.  To see that $U$ preserves fibrant objects, first note that $(L^K,r)$ is a localization on $\Ho\bigl(\coalgkm\bigr)$ by the first paragraph of the proof of Lemma 7.3 in \cite{crt}, whose proof works here without change.  Using the assumed statement (3), Theorem 4.6 in \cite{crt} implies that the localization $L^K$ is unique up to a natural isomorphism.  By Theorem \ref{clockwise.ccw} $L^K$ coincides with the localization $\Lcprime$.  Since the image of $\LC U$ is always fibrant in $\lcm$, $h$ implies that $U$ in \eqref{coalgklcm.forget} preserves fibrant objects.  We have shown that the first three statements are equivalent.

To see that (1) $\Longrightarrow$ (4), simply note that the morphism $c$ in Def. \ref{grso-comonad} is a weak equivalence with fibrant codomain in $\Lcprime\coalgkm$, and hence also in $\coalgklcm$.  So (1) implies that the morphism $Uc$ in $\lcm$ is a weak equivalence with fibrant codomain, i.e., a $\C$-local equivalence with $\C$-local codomain in $\M$.

Finally, we show that (4) $\Longrightarrow$ (2).  Given any $K$-coalgebra $X$, consider the functorial fibrant replacement
\[\nicexy{X \ar@{>->}[r]^-{r}_-{\sim} & \Lcprime X \ar@{>>}[r] & \ast}\]
in $\coalgklcm = \Lcprime\coalgkm$, where $*$ denotes the terminal object.  Our choice of $X \to \xtilde$ is $r : X \to \Lcprime X$.  By (4) the morphism $Ur$ in $\M$ is a $\C$-local equivalence with $\C$-local codomain.  

Next we show that $U\Lcprime X$ is weakly equivalent to $\LC UX$ in $\M$ and that there is a weak equivalence $\beta$ as in Def. \ref{def:preserves-coalgebra}(2).  Consider the commutative solid-arrow diagram 
\[\nicexy{UX \ar@{>->}[d]_-{Ur}^-{\sim} \ar[r]^-{l_{UX}}_-{\sim} & \LC UX \ar@{>>}[d]\\ U\Lcprime X \ar[r] \ar@{.>}[ur]_-{\beta} & \ast}\]
in $\lcm$, where $*$ denotes the terminal object in $\M$.  Since $r$ is a trivial cofibration in $\coalgklcm$, whose model structure is left-induced from that on $\lcm$ via $U$, the morphism $Ur$ is a trivial cofibration in $\lcm$.  The top-right composite is the fibrant replacement of $UX$ in $\lcm$.  Since the right vertical morphism is a fibration, there is a dotted lift
\[\beta : U\Lcprime X \to \LC UX\]
such that
\[\beta \circ Ur = l_{UX}.\]
Furthermore, since the left and the top morphisms are both weak equivalences in $\lcm$, so is $\beta$ by the $2$-out-of-$3$ property.  As $\beta$ is a $\C$-local equivalence between $\C$-local objects, it is actually a weak equivalence in $\M$ by \cite{hirschhorn} (Theorem 3.2.12(1)).
\end{proof}

\section{Application to Chain (Co)operadic Coalgebras, Comonoids, and (Coring) Comodules}
\label{sec:chain-applications}

In this section, we apply Theorems \ref{coalgklcm-admissible} and \ref{comonad.preservation.eq} to obtain preservation results for comonoids and (coring) comodules in chain complexes under homological truncations.

\subsection{Bousfield Localization of Chain Cooperadic Coalgebras}
\label{coop.localization}

Fix a commutative unital ring $R$.  Let $\Chr$ denote the category of unbounded chain complexes of $R$-modules. Remark \ref{remark:bounded} shows it is also possible to work with non-negatively graded chain complexes. Equip $\Chr$ with the \emph{injective model structure}, denoted $\Chrinj$, which has quasi-isomorphisms as weak equivalences and degreewise monomorphisms as cofibrations \cite{hovey} (Theorem 2.3.13).  Then $\Chrinj$ is a left proper combinatorial model category.  In particular, for each set $\C$ of cofibrations in $\Chrinj$, the left Bousfield localization $\lcchrinj$ exists and is a left proper combinatorial model category.

Suppose $Q$ is a cooperad on $\Chr$ \cite{hkrs} (Section 6.1.2).  Recall that a \emph{$Q$-coalgebra} is a pair $(X,\delta)$ consisting of an object $X$ and a structure morphism
\[\nicexy{X \ar[r]^-{\delta} & Q(n) \otimes X^{\otimes n}}\]
for each $n \geq 1$ that satisfies suitable equivariance, coassociativity, and counity conditions.  We assume that the category $\coalgqchr$ of $Q$-coalgebras is locally presentable (techniques to verify this may be found in \cite{ching-riehl} (Prop. A.1)).

Let $I$ be the interval complex in $\Chr$ \cite{bayeh} (4.1.1). Suppose there is another cooperad $P$ on $\Chr$ equipped with a morphism $Q \otimes P \to Q$ of cooperads such that $R$ is a $P$-coalgebra, and the natural factorization
\[\nicexy{R \oplus R \ar@{>->}[r] & I \ar[r]^-{\sim} & R}\]
extends to a factorization in $\coalgpchr$, in which the first morphism is a cofibration and the second morphism is a weak equivalence in $\Chrinj$.  Then by \cite{hkrs} (Theorem 6.3.1) $\coalgqchrinj$ admits the left-induced model structure via the forgetful functor $U$ in the forgetful-cofree adjunction
\[\nicexy{\coalgqchrinj \ar@<2pt>[r]^-{U} & \Chrinj \ar@<2pt>[l]^-{\Gamma_Q}}.\]

Theorem \ref{coalgklcm-admissible} applies with $\M = \Chrinj$, $\C$ an arbitrary set of cofibrations in $\Chrinj$, and $K$ the comonad $\Gamma_Q$ whose coalgebras are $Q$-coalgebras.  We conclude that there is a commutative diagram:
\[\nicexy@C+1cm{
\Chrinj \ar@{|->}[d]_-{\text{left-induced}} \ar@{|->}[r]^-{\LC} & \lcchrinj \ar@{|->}[r]^-{\text{left-induced}} & \coalgqlcchrinj \ar@{=}[d]\\
\coalgqchrinj \ar@{|->}[rr]^-{\Lcprime} && \Lcprime\coalgqchrinj}\]
In plain language, the left Bousfield localization $\Lcprime\coalgqchrinj$ exists and is equal to the left-induced model structure on $\coalgqlcchrinj$ over $\lcchrinj$.

\subsection{Bousfield Localization of Chain Comonoids}

The setting of Section \ref{coop.localization} applies, in particular, when $Q$ is the cooperad for comonoids, which are assumed to be non-counital, in $\Chrinj$ \cite{hkrs} (Corollary 6.3.5).  So via the forgetful-cofree adjunction
\[\nicexy{\Comonchrinj \ar@<2pt>[r]^-{U} & \Chrinj \ar@<2pt>[l]^-{\Gamma_Q}}\]
the category $\Comonchrinj$ of comonoids in $\Chrinj$ admits the left-induced model structure.  With $\C$ an arbitrary set of cofibrations in $\Chrinj$, Theorem \ref{coalgklcm-admissible} implies the existence of the following commutative diagram.
\[\nicexy@C+1cm{
\Chrinj \ar@{|->}[d]_-{\text{left-induced}} \ar@{|->}[r]^-{\LC} & \lcchrinj \ar@{|->}[r]^-{\text{left-induced}} & \Comon\bigl(\lcchrinj\bigr) \ar@{=}[d]\\
\Comonchrinj \ar@{|->}[rr]^-{\Lcprime} && \Lcprime\Comonchrinj}\]
This means the left Bousfield localization $\Lcprime\Comonchrinj$ exists and is equal to the left-induced model structure on $\Comon\bigl(\lcchrinj\bigr)$ over $\lcchrinj$.

\subsection{Homological Truncations Preserve Chain Comonoids} \label{subsec:truncations}

As above, let $R$ be a ring, and recall that for any integer $n$, $S^n(R)$ denotes the chain complex that is $R$ in degree $n$ and 0 elsewhere \cite{hovey} (2.3.3). On $\Chrinj$, let $L_n$ be the left Bousfield localization with respect to the set $\mathcal{C}_n$ consisting of the unique morphism $c: S^{n+1} \to 0$. It is easy to verify that the $L_n$-local weak equivalences are the chain maps $f$ such that $H_{\leq n}f$ is an isomorphism and that an $L_n$-local object is a chain complex $X$ with $H_{>n}X = 0$, just like the Postnikov section in topology. We refer to $L_n$ as the \emph{homological truncation above $n$}. For a chain complex $X = \bigl\{X_i, d_i\bigr\}$, its homological truncation above $n$ can be explicitly described as
\[(L_nX)_i = \begin{cases} X_i & \text{ if $i \leq n$};\\
X_{n+1}/\ker d_{n+1} & \text{ if $i=n+1$};\\ 0 & \text{ if $i > n+1$}.\end{cases}\]
Here $d_{n+1} : X_{n+1} \to X_n$, and the differentials in $L_nX$ are induced by those in $X$.  The $L_n$-localization map of $X$ is the quotient map $q_n : X \to L_nX$.

We now observe that if $n < -1$, then the homological truncation above $n$ preserves comonoids, so all four conditions in Theorem \ref{comonad.preservation.eq} hold.  To see this, suppose $(X,\delta)$ is a comonoid in $\Chrinj$.  We define a comultiplication $\delta_n : L_nX \to (L_nX)^{\otimes 2}$ by
\[\delta_n(x) = \begin{cases} q_n^{\otimes 2}(\delta x) & \text{ if $|x| \leq n$};\\ 
0 & \text{ if $|x| \geq n+1$}\end{cases}\]
where $|x|$ means the degree of a homogeneous element $x$.  To see that this map gives $L_nX$ the structure of a comonoid, we first observe that the square
\begin{equation}\label{chain-comonoid-truncation}
\nicexy{X \ar[r]^-{q_n} \ar[d]_-{\delta} & L_nX \ar[d]^-{\delta_n}\\
X^{\otimes 2} \ar[r]^-{q_n^{\otimes 2}} & (L_nX)^{\otimes 2}}\end{equation}
is commutative.  This is true by definition for $x \in (L_nX)_{\leq n}$.  It remains to show that
\begin{equation}\label{truncate-delta-x}
q_n^{\otimes 2}(\delta x) = 0 \ifspace |x| \geq n+1.\end{equation}
Writing $\delta(x) = \sum_{i \in \mathbb{Z}} x_i \otimes x_{|x|-i} \in X^{\otimes 2}$, we have
\begin{equation}\label{qdeltax}
q_n^{\otimes 2} (\delta x) = \sum_{i \in \mathbb{Z}} q_nx_i \otimes q_nx_{|x|-i}.\end{equation}
By definition $q_nx_i = 0$ for $i>n+1$.  Furthermore, if $|x| \geq n+1 \geq i$, then $|x|-i \geq 0$, so $q_nx_{|x|-i}=0$ because $n<-1$.  This proves \eqref{truncate-delta-x}.

The coassociativity of the map $\delta_n$ is now a formal consequence of that of $\delta$ (i.e., $(\delta\otimes\Id)\delta=(\Id\otimes\delta)\delta$) and the commutativity of the square \eqref{chain-comonoid-truncation} (i.e., $\delta_nq_n = q_n^{\otimes 2}\delta$).  Indeed, we only need to prove the coassociativity of $\delta_n$ starting with an element $x \in (L_nX)_{\leq n}$, and this follows from the following computation:
\begin{equation}\label{deltan-coassociative}
\begin{split} &(\delta_n \otimes \Id)(\delta_n x) = (\delta_n \otimes \Id)(q_n^{\otimes 2})(\delta x)
= (\delta_n q_n \otimes q_n)(\delta x) = (q_n^{\otimes 2} \delta \otimes q_n)(\delta x)\\ 
&= q_n^{\otimes 3} (\delta \otimes \Id)(\delta x) = q_n^{\otimes 3} (\Id\otimes\delta)(\delta x)
= (q_n \otimes q_n^{\otimes 2}\delta)(\delta x) = (q_n \otimes \delta_nq_n)(\delta x)\\
&= (\Id \otimes \delta_n)(q_n^{\otimes 2})(\delta x) = (\Id \otimes \delta_n)(\delta_n x).
\end{split}\end{equation}
Therefore, $(L_nX,\delta_n)$ is a comonoid.

The commutativity of the square \eqref{chain-comonoid-truncation} now implies the $L_n$-localization map
\[q_n : (X,\delta) \to (L_nX,\delta_n)\]
is a map of comonoids.  So for $n<-1$, the homological truncation $L_n$ above $n$ lifts to the homotopy category of comonoids, and all four conditions in Theorem \ref{comonad.preservation.eq} hold.

\subsection{Bousfield Localization of Chain Comodules}

Suppose $B$ is a comonoid in $\Chr$, and $\Chrb$ is the category of right $B$-comodules \cite{doi}.  The setting of Section \ref{coop.localization} applies when $Q$ is the cooperad for right $B$-comodules \cite{hkrs} (Corollary 6.3.7).  So via the forgetful-cofree adjunction
\[\nicexy{\Chrinjb \ar@<2pt>[r]^-{U} & \Chrinj \ar@<2pt>[l]^-{\Gamma_Q}}\]
the category $\Chrinjb$ of right $B$-comodules in $\Chrinj$ admits the left-induced model structure.  With $\C$ an arbitrary set of cofibrations in $\Chrinj$, Theorem \ref{coalgklcm-admissible} implies the existence of the following commutative diagram.
\[\nicexy@C+1cm{
\Chrinj \ar@{|->}[d]_-{\text{left-induced}} \ar@{|->}[r]^-{\LC} & \lcchrinj \ar@{|->}[r]^-{\text{left-induced}} & \bigl(\lcchrinj\bigr)^B \ar@{=}[d]\\
\Chrinjb \ar@{|->}[rr]^-{\Lcprime} && \Lcprime\Chrinjb}\]
This means the left Bousfield localization $\Lcprime\Chrinjb$ exists and is equal to the left-induced model structure on $\bigl(\lcchrinj\bigr)^B$ over $\lcchrinj$.

\subsection{Homological Truncations Preserve Chain Comodules}\label{sec:pres-chain-comodules}

Suppose $B$ is a comonoid such that $B_j=0$ for $j \geq 0$.  We now observe that for each integer $n$, the homological truncation $L_n$ above $n$ preserves right $B$-comodules.  

To prove this, suppose $(M,\delta)$ is a right $B$-comodule with structure map $\delta : M \to M \otimes B$.  We define a right $B$-coaction $\delta_n : L_nM \to L_nM \otimes B$ by
\[\delta_n(x) = \begin{cases} (q_n \otimes \Id_B)(\delta x) & \text{ if $|x| \leq n$};\\ 
0 & \text{ if $|x| \geq n+1$}\end{cases}\]
for $x \in L_nM$.  An argument similar to \eqref{qdeltax} shows the square
\[\nicexy@C+.3cm{M \ar[r]^-{q_n} \ar[d]_-{\delta} & L_nM \ar[d]^-{\delta_n}\\
M \otimes B \ar[r]^-{q_n\otimes \Id_B} & L_nM\otimes B}\]
is commutative.  Using the commutativity of this square, a computation similar to \eqref{deltan-coassociative} proves the coassociativity of $\delta_n$.  

The commutativity of the previous square now implies the $L_n$-localization map
\[q_n : (M,\delta) \to (L_nM,\delta_n)\]
is a map of right $B$-comodules.  So for a comonoid $B$ with $B_{\geq 0}=0$ and an arbitrary integer $n$, the homological truncation $L_n$ above $n$ lifts to the homotopy category of right $B$-comodules, and all four conditions in Theorem \ref{comonad.preservation.eq} hold.

\subsection{Bousfield Localization of Coring Comodules}

Suppose $A$ is a monoid in $\Chr$, and $\Moda$ is the category of right $A$-modules.  Via the forgetful-Hom adjunction
\[\nicexy@C+1cm{\Moda \ar@<2pt>[r]^-{U} & \Chrinj \ar@<2pt>[l]^-{\Hom(A,-)}}\]
the category $\Moda$ admits a left-induced left-proper, combinatorial model structure that coincides with the usual right-induced model structure \cite{hkrs} (2.2.3 and 6.6) and \cite{hess-shipley} (3.11).  In particular, for each set $\C$ of cofibrations in $\Moda$ (i.e., maps in $\Moda$ that are underlying cofibrations in $\Chrinj$), the left Bousfield localization $\lcmoda$ exists and is a left proper, combinatorial model category.

Suppose $B$ is an $A$-coring, which means that it is a comonoid in the monoidal category of $(A,A)$-bimodules, and $\Modab$ is the category of right $B$-comodules in right $A$-modules.  There is a forgetful-cofree adjunction
\begin{equation}\label{fcofree-coring-comodule}
\nicexy@C+.5cm{\Modab \ar@<2pt>[r]^-{U} & \Moda \ar@<2pt>[l]^-{- \otimes_A B}}.
\end{equation}
The category $\Modab$ is locally presentable, and it admits the left-induced model structure via $U$ in \eqref{fcofree-coring-comodule} \cite{hkrs} (6.6.3).

Theorem \ref{coalgklcm-admissible} applies with $\M = \Moda$, $\C$ an arbitrary set of cofibrations in $\Moda$, and $K$ the comonad whose category of coalgebras is $\Modab$.  In other words, we have the following commutative diagram.
\begin{equation}\label{comodule-localization-commute}
\nicexy@C+1cm{
\Moda \ar@{|->}[d]_-{\text{left-induced}} \ar@{|->}[r]^-{\LC} & \lcmoda \ar@{|->}[r]^-{\text{left-induced}} & \lcmodab \ar@{=}[d]\\
\Modab \ar@{|->}[rr]^-{\Lcprime} && \Lcprime\Modab}\end{equation}
This means the left Bousfield localization $\Lcprime\Modab$ exists and is equal to the left-induced model structure on $\lcmodab$ over $\lcmoda$.

\subsection{Homological Truncations Preserve Coring Comodules}

Suppose $A$ is a non-unital monoid satisfying $A_{\leq 0}=0$, and $B$ is a non-counital $A$-coring satisfying $B_{\geq 0}=0$.  We now observe that for each integer $n$, the homological truncation $L_n$ above $n$ preserves right $B$-comodules in $\Moda$.

Indeed, we simply reuse the arguments in Section \ref{sec:pres-chain-comodules}.  For a right $B$-comodule $(M,\delta,\mu)$ in $\Moda$ with right $A$-action $\mu : M \otimes A \to M$, we define a right $A$-action $\mu_n : L_nM \otimes A \to L_nM$ by
\[\mu_n(x,a) = \begin{cases} q_n\mu(x,a) & \text{ if $|x| \leq n$};\\ 
0 & \text{ if $|x| \geq n+1$}\end{cases}\]
for $x \in L_nM$ and $a \in A$.  If $|x| \geq n+1$, then
\[q_n\mu(x,a) = 0 \forallspace a \in A\]
since $|a|>0$ implies $|\mu(x,a)|>n+1$.  So the square
\[\nicexy@C+.3cm{M \otimes A \ar[r]^-{q_n \otimes \Id_A} \ar[d]_-{\mu} & L_nM \otimes A \ar[d]^-{\mu_n}\\
M \ar[r]^-{q_n} & L_nM}\]
is commutative.  Using the commutativity of this square, a computation similar to \eqref{deltan-coassociative} proves the associativity of $\mu_n$.  

The commutativity of the previous square now implies the $L_n$-localization map
\[q_n : (M,\delta,\mu) \to (L_nM,\delta_n,\mu_n)\]
is a map in $\Moda$, hence in right $B$-comodules in $\Moda$ when combined with the arguments in Section \ref{sec:pres-chain-comodules}.  So for a non-unital monoid $A$ with $A_{\leq 0}=0$, a non-counital $A$-coring $B$ with $B_{\geq 0}=0$, and an arbitrary integer $n$, the homological truncation $L_n$ above $n$ lifts to the homotopy category of right $B$-comodules in $\Moda$, and all four conditions in Theorem \ref{comonad.preservation.eq} hold.

\begin{remark} \label{remark:bounded}
Throughout this section, we could have also worked with non-negatively graded chain complexes \cite{hkrs} (6.5, 6.6). The injective model on $\M = \Chr^{\geq 0}$ is combinatorial and has all objects cofibrant, hence is left proper. When $R$ is a field, there is model structure on coassociative, counital coalgebras \cite{sore} (Thm. 3.3) that is left-induced from $\M$, originally due to Getzler and Goerss. Theorem \ref{coalgklcm-admissible} applies in this setting for any left Bousfield localization. The left Bousfield localizations of bounded chain complexes are characterized in \cite{white-localization} (7.5).
Early results of Hess, related to Hopf-Galois extensions and (co)descent, used the setting of non-negatively graded chain complexes, and produced model structures on comodules and coalgebras, but were subsumed by \cite{bayeh} and \cite{hkrs}.
In this context of bounded chain complexes, Theorem \ref{comonad.preservation.eq} can be used to lift localizations from categories of algebras and modules (which can themselves be lifted from chain complexes using \cite{batanin-white,white-yau}) to categories of $H$-coalgebras or $V$-comodules. 

\end{remark}

The category of non-negatively graded cochain complexes with the injective model structure is combinatorial and has all objects cofibrant \cite{white-yau2} (Section 11), so the first two conditions of Theorem \ref{coalgklcm-admissible} are satisfied. 

\begin{problem}
Work out when the categories of $K$-coalgebras discussed in this section admit a left-induced model structure in the context of cochain complexes, and then apply Theorems \ref{coalgklcm-admissible} and \ref{comonad.preservation.eq} to deduce preservation results for coalgebras under left Bousfield localization.
\end{problem}

There is also a Hurewicz model structure on chain complexes, where weak equivalences are chain homotopy equivalences. In general, this model structure is not cofibrantly generated in the classical sense but is \textit{enriched cofibrantly generated} \cite{hkrs} (4.1). Hence, the first condition of Theorem \ref{coalgklcm-admissible} does not apply. Nevertheless, categories of coalgebras over comonads often admit left-induced model structures in this setting \cite{hkrs} (6.3.7, 6.6.3).

\begin{problem}
Generalize Theorems \ref{coalgklcm-admissible} and \ref{comonad.preservation.eq} to the setting of enriched cofibrant generation and deduce preservation results for coalgebras in the Hurewicz model structures under left Bousfield localization.
\end{problem}

\begin{remark}
There are many interesting examples of left Bousfield localizations on chain complexes, beyond the homological truncation example above. First, any exact colocalization gives rise to a localization \cite{white-yau2} (11.3), so the colocalization (right Bousfield localization) examples of \cite{white-yau2} (Sec. 11) yield left Bousfield localizations. More generally, the local duality framework of \cite{valenzuela1} (3.2) provides many examples. For example, if $A$ is a ring and $I\subset A$ is a finitely generated ideal, then there are smashing localization and colocalization functors, for $I$-localization and $I$-completion \cite{valenzuela1} (Thm. 3.11). These local duality localizations are Bousfield localizations with respect to a set of morphisms, e.g., $I$-localization corresponds to the homology theory $M\mapsto H_*(R/I \otimes_R M)$ and hence to a left Bousfield localization at a set of morphisms constructed using the Bousfield-Smith cardinality argument (see Section \ref{sec:smashing-comodule}). Since our results above imply that every left Bousfield localization preserves coalgebra structure in these algebraic contexts, Theorem \ref{comonad.preservation.eq} applies to these examples.
\end{remark}

\subsection{Model structures on coalgebras over operads}

Let $R$ be a semi-simple ring. Let $P$ be an operad in $\Chr$. A chain complex $X$ is a co-algebra over $P$ if there is a morphism of operads from $P$ to the coendomorphism operad of $X$, defined by $\CoEnd_X(n) = \Hom(X,X^{\otimes n})$.

If $P$ is a planar operad, or a cofibrant operad, or if $P$ is the linear analogue of the Barratt-Eccles operad, then there is a model structure on the category of $P$-coalgebras left-induced from the injective model structure on $\Chr$ \cite{brice2} (8.9). Hence, we say that $P$ is \textit{coadmissible}. This result builds on earlier work of Getzler and Goerss, which worked with the associative operad. 

In fact, left-induced model structure of \cite{brice2} (8.9) exists slightly more generally, in $\ChA$ where $\cat A$ is a semi-simple abelian category satisfying certain categorical conditions (e.g., the rewrite conditions) required to make the category of $P$-coalgebras comonadic \cite{brice2} (5.14). However, these left induced model structures do not exist for arbitrary $P$, e.g., the Com operad is not coadmissible even if $R$ is an algebraically closed field of characteristic zero \cite{brice2} (8.10).

Analogously to the examples above, one can now formulate preservation results for these categories of coalgebras under various left Bousfield localizations such as truncations or local duality localizations, using Theorem \ref{comonad.preservation.eq}. This gives another way to study coalgebras over cooperads, since the linear dual of an operad is a cooperad with the same coalgebras, and vice versa.

\subsection{Model Category of Cooperads}

Let $k$ be a field. Endow the category of bounded-below chain complexes of $k$-vector spaces with the levelwise injective model structure. Then, there is a left-induced model structure on the category of unital cooperads \cite{aubry} (Thm. 2.4.1).

Let $F$ be the comonad dual to the free operad functor. Let $P$ be a quasi-cofree unital $F$-comonoid. Then the category of $P$-coalgebras has a model structure left-induced from the injective model structure on unbounded chain complexes of $k$-vector spaces \cite{aubry} (Thm. 3.2.3).

Now let $k$ be a field of characteristic zero. The Hinich model structure on dg-operads has quasi-isomorphisms as weak equivalences and surjections as fibrations. The bar-cobar adjunction may be used to left-induce a model structure on the category of curved conilpotent cooperads, from this model structure on dg operads \cite{brice2} (Theorem 7.15). Note that this result is not true if we shift from dg operads to dg-categories. 

In all three of these settings, the machinery of Section \ref{coop.localization} applies, and hence all coalgebras are preserved by any left Bousfield localization. In particular, Theorem \ref{coalgklcm-admissible} produces left-induced model structures on coalgebras after truncation (\ref{subsec:truncations}) and Theorem \ref{comonad.preservation.eq} characterizes when truncations preserve coalgebras. Truncations in operadic contexts were critically important in a recent proof of the Baez-Dolan stabilization hypothesis \cite{Reedy-paper}.

\subsection{Koszul duality}

Another way to left-induce model structures on coalgebras relies on Koszul duality contexts. Classically, Quillen equivalences relate associative algebras and coassociative coalgebras (since the associative operad is its own Koszul dual), and relate dg Lie algebras and cocommutative dg coalgebras (since the commutative operad is Koszul dual to the Lie operad). 

For simplicity, we work over a field $k$ of characteristic zero. The basic idea of Koszul duality is to start with an operad $P$ and construct a cooperad $BP$ via the bar construction, such that $P$-algebras and $BP$-coalgebras are related by the bar-cobar adjunction. Often, model structures can be left induced across the bar-cobar adjunction, and this provides numerous settings for Theorem \ref{comonad.preservation.eq} to apply.

The bar-cobar adjunction leads to a number of left-induced model structures on coalgebras, that we discuss below. Our main references, \cite{hirsh, brice0}, subsume the work of many authors, which we briefly recall. Quillen produced a model structure (for the case $k = \mathbb{Q}$) on the category of simply connected cocommutative dg coalgebras, left induced along the cobar functor to dg Lie algebras. 
Hinich extended this to the category of unital coalgebras, again left induced from dg Lie algebras, 
and to the category of conilpotent cocommutative coalgebras.
Lefevre-Hasegawa applied the same method to produce a left-induced model structure for coassociative coalgebras. 
Vallette then extended this method to the context of Koszul operads $P$, by which he means operads that admit a quadratic presentation by generators and relations satisfying certain axioms 
that allow him to construct the Koszul dual cooperad (which he denotes $P^{\hspace{1pt}\ac}$) such that the cobar construction of $P^{\hspace{1pt}\ac}$ is weakly equivalent to $P$. He then produces a model structure on conilpotent $P^{\hspace{1pt}\ac}$-coalgebras, left-induced from the usual model structure on $P$-algebras (which is right-induced from $\Chk$), for which the bar-cobar adjunction is a Quillen equivalence. As pointed out in \cite{brice0}, Vallette's conditions force $P$ to be augmented, which rules out operads that describe algebras with units. 

All of these results were subsumed by \cite{hirsh}, who assume that $P$ is an augmented operad and $C$ is a conilpotent coaugmented weight-graded cooperad, together with a morphism $\alpha: P\to \Omega C$ of operads (where $\Omega C$ is the cobar construction of the cooperad $C$, and $\alpha$ is called a \textit{twisting morphism from $C$ to $P$}). If either $k$ has characteristic zero or $P$ and $\Omega C$ are $\Sigma$-split operads, then the category of conilpotent $C$-coalgebras has a model structure left-induced from the category of $P$-algebras along the cobar functor $\Omega_\alpha$ associated to $\alpha$ \cite{hirsh} (Theorem 3.11). Note that $C$ is now not assumed to be dual to $P$ in any way, and a particular choice of $P$ and $\alpha$ yields model structures left-induced from $\Chk$ along a functor $\Omega_\epsilon$, in particular recovering earlier results of Positselski on conilpotent dg-coalgebras. There are many twisting morphisms, so this procedure produces many model structures on coalgebras, e.g., the authors list three model structures on conilpotent coassociative coalgebras (and on categories of comodules), coming from different twisting morphisms and with different classes of weak equivalences \cite{hirsh} (Section 4). Furthermore, \cite{hirsh} (Prop. 3.15) gives conditions under which $\Omega_\alpha$ is a Quillen equivalence, and \cite{hirsh} (Remark 5.13) proves that these model structures exist in the case of bounded as well as unbounded chain complexes.

An alternative generalization of Vallette's results, to the setting of curved conilpotent cooperads $C$, appears in \cite{brice0}. For any such $C$, an \textit{operadic twisting morphism} $\alpha$ is defined to be a degree $-1$ map from the symmetric collection underlying $C$ to the symmetric collection underlying $P$ \cite{brice0} (Def. 62). In the setting of $\Chk$ where $k$ is a field of characteristic zero, there is a model structure on $C$-coalgebras left-induced from $P$-algebras along the cobar functor $\Omega_\alpha$ associated to $\alpha$ \cite{brice0} (Theorem 82). The curved setting allows for the study of algebras with units, and recovers model structures on curved coalgebras previously studied by Positselski. Conditions under which $\Omega_\alpha$ is a Quillen equivalence are provided in \cite{brice0} (Theorem 110). For non-symmetric (co)algebras, the characteristic zero assumption can be dropped \cite{brice0} (Theorem 82). As a special case we obtain left-induced model structures on categories of noncounital curved conilpotent coassociative coalgebras \cite{brice0} (Prop. 123), on curved conilpotent coalgebras \cite{brice0} (Remark 126), and on curved conilpotent Lie coalgebras \cite{brice0} (Prop. 129).

These model structures on coalgebras over (curved) co-operads are connected via left Bousfield localizations \cite{hirsh} (Theorem 3.11), \cite{brice0} (Prop. 96). Furthermore, our previous papers \cite{batanin-white, white-localization, white-yau} have shown how to lift left Bousfield localizations to categories of $P$-algebras. Hence, we can apply Theorem \ref{coalgklcm-admissible} to the model category $\M$ of $P$-algebras, to verify the condition of Theorem \ref{comonad.preservation.eq}, and thereby further lift them to categories of $C$-coalgebras, along any cobar functor $\Omega_\alpha$. We can also lift localizations from $\Chk$ to $C$-coalgebras along the functors $\Omega_\epsilon$ discussed above. In cases when $\Omega_\alpha$ is a Quillen equivalence (\cite{hirsh} (Prop. 3.15), \cite{brice0} (Theorem 110)), the condition of Definition \ref{crt-comonad} is automatic for any localization of $P$-algebras lifted from $\Chk$ using \cite{batanin-white} (Theorem A), and hence the equivalent conditions of Theorem \ref{comonad.preservation.eq} are satisfied in this case, for any left Bousfield localization.

\section{Application to Comodules and Coalgebras in (Localized) Spectra}
\label{sec:spectra-applications}

In this section, we apply Theorems \ref{coalgklcm-admissible} and \ref{comonad.preservation.eq} to obtain preservation results for comodules in (localized) spectra under smashing localizations.

\subsection{Bousfield Localization of Spectral Comodules}
\label{sp.comodule.coalg}

Denote by $\Spinj$ the model category of symmetric spectra of simplicial sets with the \emph{injective} model structure; see \cite{hss} (Section 5).  In the injective model structure, cofibrations are monomorphisms, and weak equivalences are defined levelwise. Since every object is cofibrant, the combinatorial simplicial model category $\Spinj$ is also left proper.  

Suppose $A$ is a ring spectrum, i.e., a monoid in $\Spinj$.  The category $\Moda$ of right $A$-modules admits a left-induced left proper, combinatorial, simplicial model structure via the forgetful functor to $\Spinj$; see \cite{bayeh} (2.23) and \cite{hkrs} (2.2.3, 5.0.1, and 5.0.2).   Since $\Moda$ is a left proper, combinatorial, simplicial model category, for each set $\C$ of cofibrations in $\Moda$ (i.e., maps in $\Moda$ that are underlying cofibrations in $\Spinj$), the left Bousfield localization $\lcmoda$ exists and is a left proper, combinatorial, simplicial model category.

Now suppose $A$ is a strictly commutative ring spectrum, i.e., a commutative monoid in $\Spinj$, and $B$ is an $A$-coalgebra.  Denote by $\Modab$ the category of right $B$-comodules in $\Moda$.  There is a forgetful-cofree adjunction
\begin{equation}\label{forgetful-cofree-comodule}
\nicexy@C+.5cm{\Modab \ar@<2pt>[r]^-{U} & \Moda \ar@<2pt>[l]^-{- \wedge_A B}}.
\end{equation}
The category $\Modab$ is locally presentable because $\Moda$ is locally presentable and $\Modab$ is the category of coalgebras over a comonad $K_B$ on $\Moda$ that preserves colimits; see \cite{adamek-rosicky} (2.78) and \cite{ching-riehl} (Prop. A.1).  Moreover, $\Modab$ admits the left-induced model structure via $U$ in \eqref{forgetful-cofree-comodule} \cite{hkrs} (5.0.3).  

Theorem \ref{coalgklcm-admissible} applies with $\M = \Moda$, $\C$ an arbitrary set of cofibrations in $\Moda$, and $K$ the comonad whose category of coalgebras is the category $\Modab$.  We conclude that there is a commutative diagram as in \eqref{comodule-localization-commute}.  So the left Bousfield localization $\Lcprime\Modab$ exists and is equal to the left-induced model structure on $\lcmodab$ over $\lcmoda$.

\subsection{Smashing Localizations Preserve Spectral Comodules}
\label{sec:smashing-comodule}

For a spectrum $E$, recall that the $E$-localization $L_E$ \cite{bousfield} results in a model structure whose weak equivalences are the $E$-local equivalences, i.e., morphisms $f$ such that $E_*(f)$ is an isomorphism. The Bousfield-Smith cardinality argument, reproduced in \cite{hirschhorn} (2.3) produces a set of morphisms $\cat C$ such that $L_E = L_{\cat C}$ \cite{bousfield} (1.13), hence Theorem \ref{coalgklcm-admissible} can apply to these types of localizations. The localization $L_E$ is said to be \emph{smashing} if
\[L_EX \cong L_ES^0 \wedge X\]
for each spectrum $X$. In other words, $L_E$ is smashing if each localization map $l_X : X \to L_EX$ is given by applying $-\wedge X$ to the localization map $l_{S^0} : S^0 \to L_ES^0$ of the sphere spectrum.  Examples of smashing localizations include:
\begin{enumerate}
\item $L_E$ with $E$ the Moore spectrum of a torsion-free group \cite{bousfield};
\item Miller's finite localization \cite{miller};
\item $L_{E(n)}$ with $E(n)$ the $n$th Morava $E$-theory \cite{ravenel} (Theorem 7.5.6).
\end{enumerate}

With the same setting as in Section \ref{sp.comodule.coalg}, suppose $(X,\lambda)$ is a right $B$-comodule in $\Moda$.  Suppose $L_E$ is a smashing localization.  Then the $E$-localization $L_EX$ inherits a natural right $B$-comodule structure with right $A$-action
\[\nicexy@C+.4cm{L_EX \wedge A \cong L_ES^0 \wedge X \wedge A \ar[r]^-{L_ES^0\wedge \mu} &
L_ES^0 \wedge X \cong L_EX,}\]
where $\mu : X \wedge A \to X$ is the right $A$-action on $X$, and right $B$-coaction
\[\nicexy@C+.4cm{L_EX \cong L_ES^0 \wedge X \ar[r]^-{L_ES^0\wedge \lambda}
& L_ES^0 \wedge X \wedge_A B \cong L_EX \wedge_A B.}\]
Furthermore, the localization map
\[l_X = l_{S^0} \wedge X : X \to L_ES^0 \wedge X \cong L_EX\]
of $X$ respects the right $A$-module structure and the right $B$-comodule structure because the diagrams
\[\nicexy@C-.2cm{X \wedge A \ar[d]_-{\mu} \ar[r]^-{l_X \wedge A} & L_ES^0 \wedge X \wedge A \ar[d]^-{\Id \wedge \mu}\\ X \ar[r]^-{l_X} & L_ES^0 \wedge X \cong L_EX}\qquad
\nicexy{X \ar[d]_-{\lambda} \ar[r]^-{l_X} & L_ES^0 \wedge X \cong L_EX \ar[d]^-{L_ES^0\wedge \lambda}\\ X \wedge_A B \ar[r]^-{l_X \wedge_A B} & L_ES^0 \wedge X \wedge_A B \cong L_EX \wedge_A B}\]
are commutative.

Therefore, every smashing localization $L_E$ lifts to the homotopy category of $K_B$-coalgebras in the sense of Def. \ref{crt-comonad}, where $K_B$ is the comonad for right $B$-comodules in $\Moda$.  By Theorem \ref{comonad.preservation.eq} all four conditions there hold.  In particular, every smashing localization $L_E$ preserves right $B$-comodules.

\subsection{Smashing Localizations on Localized Spectra Preserve Comodules}

Fix a prime $p$, and suppose $\sS$ is the $p$-localization of $\Spinj$. For an arbitrary but fixed $E \in \sS$, suppose $\M$ is the $E$-localization of $\sS$. We do not assume that $L_E$ is smashing. The category $\M$ is still a simplicial combinatorial model category with all objects cofibrant by Remark \ref{remark:cof-between-cof-obj}. Well-studied examples of such localized categories of spectra $\M$ and smashing localizations on $\M$ include the following cases:
\begin{itemize}
\item $E$ is the Morava theories $E(n)$ or $K(n)$ \cite{hovey-strickland};
\item $E$ is the wedge $\bigvee_{n \geq 0} K(n)$, the Eilenberg-Mac Lane spectrum $H\mathbb{F}_p$, or the Brown-Comenetz dual $IS^0$ of the sphere spectrum \cite{wolcott}.
\end{itemize}

For a comonoid $B$ in $\M$, the same discussion as in Sections \ref{sp.comodule.coalg} and \ref{sec:smashing-comodule} apply.  Indeed, the same proof as in \cite{hkrs} Theorem 5.0.3 (which deals with $\Spinj$ instead of $\M$) implies the category $\Comod(B;\M)$ of right $B$-comodules in $\M$ admits a left-induced model structure via the forgetful-cofree adjunction
\begin{equation}\label{comodule-forget-cofree}
\nicexy{\Comod(B;\M) \ar@<2pt>[r]^-{U} & \M. \ar@<2pt>[l]^-{-\wedge B}}\end{equation}
Theorem \ref{coalgklcm-admissible} now applies with $\C$ an arbitrary set of cofibrations in $\M$ and with $K$ the comonad whose category of coalgebras is $\Comod(B;\M)$.  We conclude that there is a commutative diagram
\begin{equation}\label{comodule-localization}
\nicexy@C+1cm{
\M \ar@{|->}[d]_-{\text{left-induced}} \ar@{|->}[r]^-{\LC} & \lcm \ar@{|->}[r]^-{\text{left-induced}} & \Comod(B;\lcm) \ar@{=}[d]\\
\Comod(B;\M) \ar@{|->}[rr]^-{\Lcprime} && \Lcprime\Comod(B;\M).}\end{equation}
So the left Bousfield localization $\Lcprime\Comod(B;\M)$ exists and is equal to the left-induced model structure on $\Comod(B;\lcm)$ over $\lcm$.

The exact same discussion as in Section \ref{sec:smashing-comodule} implies every smashing localization on $\M$ lifts to the homotopy category of $K$-coalgebras in the sense of Def. \ref{crt-comonad}, where $K$ is the comonad for right $B$-comodules in $\M$.  By Theorem \ref{comonad.preservation.eq} all four conditions there hold, so every smashing localization  preserves right $B$-comodules in $\M$.

\begin{remark} \label{remark:local-duality}
Local duality provides another example of a smashing localization of spectra \cite{valenzuela}. While \cite{valenzuela} is written in the context of stable $\infty$-categories, \cite{valenzuela1} points out that the results also hold in the combinatorial stable model categories we consider. If $R$ is a commutative ring spectrum (indeed, $E_2$ is enough), and $\mathfrak{p}$ is a finitely generated and homogeneous prime ideal of $\pi_* R$, then there is a smashing localization $L_{\mathfrak{p}}$ of $R$-modules with the property that $\pi_*(L_{\mathfrak{p}} M) \cong (\pi_* M)_{\mathfrak{p}}$ by \cite{valenzuela} (Prop. 3.4) or \cite{ekmm} (Ch. V.1) at the model categorical level. When $\pi_* R$ is Noetherian, this can be extended to any specialization-closed subset $\cat V$ of $\Spec(\pi_* R)$ \cite{valenzuela} (Thm. 3.9). More generally, any abstract local duality context where one localizes with respect to a localizing ideal \cite{valenzuela} (Lemma 2.7) yields a smashing localization. As shown above, these smashing localizations preserve comodule structure, and more generally preserve $K$-coalgebra structure in chain complexes (Sec.\ \ref{sec:chain-applications}) and in the stable module category (Sec.\ \ref{sec:application-stmod}) thanks to the work of \cite{valenzuela}.
\end{remark}

\begin{remark}
In this section, we have focused on smashing localizations, but non-smashing localizations can also preserve comonad structure. For example, the following left Bousfield localizations are not smashing \cite{hovey-strickland} (8.1):
\begin{enumerate}
\item localization $L_{K(n)}$ with respect to Morava $K$-theory,
\item harmonic localization, and 
\item $p$-completion.
\end{enumerate}
However, in many situations, a comodule over a comonoid remains so after $p$-completion, so Theorem \ref{comonad.preservation.eq} applies. Many more examples of smashing and non-smashing localizations of spectra are cataloged in \cite{white-localization}. Furthermore, \cite{white-yau2} (Sec. 11) discusses Bauer's technique of lifting localizations and colocalizations from the category of chain complexes to spectra. Lastly, thanks to cofiber sequences, every exact colocalization gives rise to a localization, so the colocalizations of \cite{white-yau2} also provide interesting examples where one can attempt to use the machinery of Theorem \ref{comonad.preservation.eq} to lift left Bousfield localizations to $K$-coalgebras.
\end{remark}

\subsection{Spectral coalgebras}

Let $A$ be a commutative symmetric spectrum. There are combinatorial model structures on the category of $A$-coalgebras and on the category of cocommutative $A$-coalgebras, both left induced from the injective stable model structure on $A$-modules \cite{peroux22} (Prop. 5.2, 5.3). This is a natural setting for computations of topological coHochschild homology. It would be natural to compute a left Bousfield localization with respect to the $coTHH_*$-isomorphisms. It would also be natural to lift smashing localizations to categories of coalgebras, using the techniques of Section \ref{sec:equivalent-approach}. We leave this application to the interested reader.

\subsection{Projective model structure}

Left-induced model structures on spectra and localized spectra play a role in studying Waldhausen $K$-theory via comodules. Let $\Sppr$ denote the projective model structure, and let $\Sp$ denote the stable projective model structure, on the category of symmetric spectra. For a generalized homology theory $E_*$, let $\SpprE$ and $\Sp_E$ denote the $E$-local model structures, obtained by left Bousfield localization. If $X$ is a simplicial set, then $\Sigma^\infty X_+$ is a coalgebra (thanks to the diagonal map), and the category of spectral comodules has a model structure left induced from $\Sppr$ or $\SpprE$ \cite{wald} (Prop. 5.17) or from $\Sp$ or $\Sp_E$ \cite{wald} (Thm. 5.2). 
There are also twisted homology localizations \cite{wald} (Prop. 5.19), for which the set of morphisms $\cat C$ to invert may be taken to be the set of generating trivial cofibrations from \cite{wald} (Thm. 4.10) for pointed simplicial sets, or its stabilization for spectra. Hence, the model structure of \cite{wald} (Prop. 5.19) may be constructed as a left Bousfield localization of the model structure on spectral comodules \cite{wald} (Thm. 5.2). Lastly, if $H$ is a simplicial monoid then there is a model structure on the category of $\Sigma^\infty H_+$-comodules in the category of ring spectra, left induced from the $E$-local model structure on ring spectra \cite{wald} (Cor. 5.7). All of these settings satisfy the conditions of Theorem \ref{clockwise.ccw}, where the set of morphisms $\cat C$ is constructed using the Bousfield-Smith cardinality argument as in Section \ref{sec:smashing-comodule}. Thus, we can use Theorem \ref{comonad.preservation.eq} to obtain preservation and localization-lifting results in projective settings, as we did above for the injective model structure.

\subsection{Motivic and equivariant spectra}

As far as the authors know, left-induced model structures on categories of coalgebras have not been produced in motivic or equivariant contexts. Hence, this subsection focuses on open problems that may be of interest to the reader.

\begin{problem}
Extend the proof of \cite{bayeh} (Thm.\ 5.0.3) to produce left-induced model structures on comodules over a coalgebra in motivic symmetric spectra, by replacing $sSet$ by motivic spaces, and using the good cylinder object therein.
\end{problem}

The existence of such a left-induced model structure appears likely. With it in hand, one can ask questions regarding the preservation of $K$-coalgebras under left Bousfield localization.

\begin{problem}
Use the approach of Section \ref{sec:preservation-coalgebra} to prove preservation results for motivic comodules (and coalgebras over cooperads) under left Bousfield localizations as we have done above for symmetric spectra.
\end{problem}

The situation for equivariant spectra is more difficult. First, most of the literature on equivariant spectra begins with the category of compactly generated spaces, which is not locally presentable. In \cite{white-yau6} (Sec. 6), we formulated a number of open problems regarding the existence of useful model structures on categories of equivariant spectra, including injective model structures on global equivariant spectra, on Hausmann's $G$-symmetric spectra, and on Mandell's equivariant symmetric spectra. In \cite{white-localization} (Lemma 5.60), the first author produced injective model structures on $G$-equivariant orthogonal spectra based on the locally presentable category of $\Delta$-generated spaces. We are therefore in a position to pose the following problem.

\begin{problem}
Extend the proof of \cite{bayeh} (Thm.\ 5.0.3) to produce left-induced model structures on comodules over a coalgebra in equivariant spectra, based on simplicial sets or on $\Delta$-generated spaces, with any of the injective model structures just listed.
\end{problem}

If this can be done, then there is a natural follow-up question to determine which left Bousfield localizations of equivariant spectra (cataloged in \cite{gutierrez-white-equivariant, white-localization} among other places) preserve $K$-coalgebra structure.

\begin{problem}
Use the approach of Section \ref{sec:preservation-coalgebra} to prove preservation results for comodules (and coalgebras over cooperads) in equivariant spectra under left Bousfield localizations as we have done above for symmetric spectra.
\end{problem}

Lastly, in both the motivic and equivariant setting, it would be natural to study coalgebras as well as comodules.

\begin{problem}
Extend P\'{e}roux's model structure on (cocommutative) spectral coalgebras \cite{peroux22} (Prop. 5.2, 5.3) to the injective model structures on motivic symmetric spectra, and on equivariant spectra discussed above, and then formulate preservation results for left Bousfield localization, following Section \ref{sec:equivalent-approach}.
\end{problem}

\section{Application to Comodules and Cooperadic Coalgebras in the Stable Module Category}
\label{sec:application-stmod}

In this section, we apply Theorems \ref{coalgklcm-admissible} and \ref{comonad.preservation.eq} to obtain preservation results for comodules and cooperadic coalgebras in the stable module category under smashing localizations.

\subsection{Smashing Localizations Preserve Comodules}

Another adaptation of Section \ref{sec:spectra-applications} applies to the stable module category.  Suppose $\M$ is the stable module category of $kG$-modules, where $k$ is a field whose characteristic divides the order of the finite group $G$, equipped with the cofibrantly generated model structure in \cite{hovey} (Section 2.2). It is possible to work over more general rings, as Remark \ref{remark:stanculescu} illustrates.

Suppose $B$ is a comonoid in $\M$.  The same proof of \cite{hkrs} (Corollary 6.3.7), which deals with $\Chr$ instead of $\M$, shows the category $\Comod(B;\M)$ of right $B$-comodules in $\M$ admits a left-induced model structure via the forgetful-cofree adjunction as in \eqref{comodule-forget-cofree}.  Theorem \ref{coalgklcm-admissible} now applies with $\C$ an arbitrary set of cofibrations in $\M$ and with $K$ the comonad whose category of coalgebras is $\Comod(B;\M)$.  So there is a commutative diagram as in \eqref{comodule-localization}.

Suppose $L_E$ is a smashing localization on the stable module category $\M$, so
\[L_EY \cong L_E\tensorunit \otimes Y\]
for each $Y \in \M$, where $\tensorunit$ is the monoidal unit in $\M$. 
The exact same discussion as in Section \ref{sec:smashing-comodule} implies $L_E$ lifts to the homotopy category of $K$-coalgebras, where $K$ is the comonad for right $B$-comodules in $\M$.  By Theorem \ref{comonad.preservation.eq} all four conditions there hold, so every smashing localization preserves right $B$-comodules in $\M$.  The reader is referred to \cite{benson-iyengar-krause-stratifying-localizing} (Theorem 11.12) for characterizations of smashing localizations on the stable module category.

\subsection{Smashing Localizations Preserve Comonoids}
\label{sec:smashing-comonoid}

Suppose $(X,\delta)$ is a non-counital comonoid in $\M$.  Suppose $L_E$ is a smashing localization.  Then $(L_EX,\Delta)$ is a non-counital comonoid with comultiplication $\Delta$ defined as the composite:
\[\nicexy{L_EX \cong L_E\tensorunit \otimes X \ar[r]^-{\Delta} \ar[d]_-{L_E\delta\,=\, L_E\tensorunit \otimes \delta} 
& L_EX \otimes L_EX \cong L_E\tensorunit\otimes X \otimes L_E\tensorunit \otimes X\\
L_E\tensorunit \otimes X \otimes X \ar[r]^-{\cong} 
& L_E\tensorunit \otimes L_E\tensorunit \otimes X \otimes X. \ar[u]^-{\cong}_-{\text{switch}}}\]
The bottom horizontal isomorphism comes from the idempotency of the smashing localization $L_E$, i.e., 
\begin{equation}\label{localization-idempotent}
L_E\tensorunit \cong L_EL_E\tensorunit \cong L_E\tensorunit \otimes L_E\tensorunit.\end{equation}
Since the localization map is given by $l_X = l_{\tensorunit} \otimes X$, similarly to Section \ref{sec:smashing-comodule} the diagram 
\begin{equation}\label{comonoid-localization-map}
\nicexy{X \ar[d]_-{\delta} \ar[r]^-{l_X} & L_EX \cong L_E\tensorunit \otimes X \ar[d]^-{\Delta}\\
X \otimes X \ar[r]^-{l_X \otimes l_X} & L_EX \otimes L_EX}\end{equation}
is commutative.  So the localization map $l_X$ extends to a map of non-counital comonoids.

Therefore, every smashing localization $L_E$ lifts to the homotopy category of $K$-coalgebras, where $K$ is the comonad for non-counital comonoids \cite{hkrs} (Cor. 6.3.5).  By Theorem \ref{comonad.preservation.eq} all four conditions there hold.  In particular, every smashing localization on the stable module category preserves non-counital comonoids.

\subsection{Smashing Localizations Preserve Cooperadic Coalgebras}
\label{sec:smashing-operad-coalgebra}

More generally, suppose $\sO$ is a cooperad in the stable module category $\M$ satisfying $\sO(0) = 0$ \cite{hkrs} (Section 6.1.2).  Suppose $L_E$ is a smashing localization on the stable module category $\M$, and suppose $(X,\delta)$ is an $\sO$-coalgebra.  The composite
\[\nicexy@C+.3cm{L_EX \cong L_E\tensorunit \otimes X \ar[d]_-{L_E\tensorunit \otimes \delta} \ar[r]^-{\Delta} & \sO(n) \otimes (L_EX)^{\otimes n}\\
L_E\tensorunit \otimes \sO(n) \otimes X^{\otimes n} \ar[d]_-{\text{switch}}^-{\cong} 
& \sO(n) \otimes (L_EI \otimes X)^{\otimes n} \ar[u]_-{\cong} \\ 
\sO(n) \otimes L_E\tensorunit \otimes X^{\otimes n} \ar[r]^-{\cong} 
& \sO(n) \otimes (L_E\tensorunit)^{\otimes n} \otimes X^{\otimes n} \ar[u]^-{\cong}_-{\text{permute}}}\]
for each $n \geq 1$ gives the localization $L_EX$ the structure of an $\sO$-coalgebra.  The bottom horizontal isomorphism is the $n$-fold version of \eqref{localization-idempotent} and follows from the idempotency of the smashing localization $L_E$:
\[\nicexy@C+.3cm{L_E\tensorunit \cong \underbrace{L_EL_E\cdots L_E}_{n} \tensorunit \cong (L_E\tensorunit)^{\otimes n}.}\]
Similar to \eqref{comonoid-localization-map}, for each $n \geq 1$ the diagram 
\[\nicexy@C+1.4cm{X \ar[d]_-{\delta} \ar[r]^-{l_X\,=\, l_{\tensorunit} \otimes X} 
& L_EX \cong L_E\tensorunit \otimes X \ar[d]^-{L_E\tensorunit \otimes \delta}\\
\sO(n) \otimes X^{\otimes n} \ar[r]^-{\sO(n) \otimes (l_{\tensorunit} \otimes X)^{\otimes n}} 
& \sO(n) \otimes (L_E\tensorunit \otimes X)^{\otimes n} \cong L_E\tensorunit \otimes \sO(n) \otimes  X^{\otimes n}}\]
is commutative.  So the localization map $l_X$ extends to a map of $\sO$-coalgebras.

Suppose $K$ is the comonad for $\sO$-coalgebras \cite{hkrs} (Section 6.1.2).  Assume that $\coalgkm$ (resp., $\coalg(K;L_E\M)$) admits a left-induced model structure via the forgetful functor to $\M$ (resp., $L_E\M$).  Then the above discussion implies that every smashing localization lifts to the homotopy category of $K$-coalgebras.  By Theorem \ref{comonad.preservation.eq} all four conditions there hold.  In particular, every smashing localization preserves $\sO$-coalgebras whose comonad is left-admissible over $\M$ and $L_E\M$.

\begin{remark} \label{remark:stanculescu}
The results in this section may be formulated more generally. Following \cite{stanculescu} (9.2.2 and 9.2.3), we can replace $R = kG$ with any commutative von Neumann regular ring $R$, then fix a cocommutative Hopf algebra $H$ over $R$ for which the classes of projective left $H$-modules and injective left $H$-modules coincide (e.g., this occurs if $H$ is finitely generated and projective as an $R$-module), then endow the category $\M$ of left $H$-modules with the stable module category structure. Stanculescu shows that the category of left $H$-module coalgebras (i.e., comonoids in $\M$) has a left-induced model structure, and the considerations of this section apply. Furthermore, for $A$ a left $H$-module algebra, the category of $A$-corings (that is, comonoids in $(A,A)$-bimodules) has a left induced model structure for which the considerations of this section apply \cite{stanculescu} (9.2.3).
\end{remark}

Lastly, if $Q$ is a quiver (directed graph) and $\cat A$ is an abelian category, then a representation of $Q$ in $\cat A$ is a covariant functor $X: Q \to \cat A$. The stable module category is determined by a Hovey triple \cite{hovey-cotorsion} (Thm. 2.2) that can be lifted to yield a combinatorial stable model structure on the category of quiver representations \cite{quiver} (Thm. 6.3, 7.2). There are many important coalgebras in the setting of quiver representations, e.g., the coalgebra of distributions, quantum groups, and the path coalgebra of a quiver \cite{jara}. Furthermore, many localizations in this context are cataloged in \cite{jara}, which also proves that localizations preserve path coalgebras, string coalgebras, and serial coalgebras, characterizes (co)localizing subcategories of categories of comodules over a coalgebra, constructs localizations associated to any localizing subcategory (and associated to any given subset of vertices in a quiver), and lifts localizations to categories of coalgebras.

\begin{problem}
Extend the results of \cite{jara} to the model structure on quiver representations of \cite{quiver} (Thm. 6.3, 7.2). Investigate whether left-induced model structures exist on the categories of coalgebras, whether the localizations can be formulated as left Bousfield localizations, and whether the lifting result of \cite{jara} can be used to verify the conditions of Theorem \ref{comonad.preservation.eq}.
\end{problem}

\section{Application to Simplicial Settings} \label{sec:simplicial-applications}

In this section, we discuss applications of our main theorems to simplicial settings. We discuss coalgebras in spaces, a general machine for left inducing model structures on coalgebras arising from simplicial Quillen adjunctions, and simplicial presheaves.

\subsection{Spaces}

Let $\cat M$ denote either the category of compactly generated weak Hausdorff spaces or the category of simplicial sets, equipped with the Quillen model structure. In these categories, all operads are admissible (this result is well known and is spelled out in \cite{bm03, white-topological, white-commutative-monoids}). It is natural to ask whether cooperads are admissible. Every space or simplicial set $X$ has a natural cocommutative comonoid structure using the diagonal map $X \to X\times X$. This turns out to be the \textit{only} (counital) comonoidal structure on any object in any Cartesian symmetric monoidal category \cite{peroux-shipley}. Hence, there are no interesting questions about preservation of comonoid structure in these categories. Furthermore, for any comonoid $X$, the category of $X$-comodules is simply the overcategory $X/\cat M$, so there are no interesting questions about model structures and preservation under localization for comodules. 

Similarly, if $H$ is a bimonoid in $\cat M$ then the category of $H$-comodule algebras in $\cat M$ has a model structure left-induced from the usual model structure on monoids and related to Hopf-Galois extensions and (co)descent \cite{hess}(1.3.1). Hence, Theorem \ref{comonad.preservation.eq} can be used to lift localizations from monoids (which can themselves be lifted from spaces using \cite{batanin-white, white-yau}) to $H$-comodule algebras. However, this can be done more directly thanks to the observation that the category of $H$-comodule algebras is isomorphic to the category of monoids over $H$.

However, there are still interesting structures possible as coalgebras over other cooperads. For example, the CDC cooperad \cite{walter} (5.2) is related to Hatcher's $\Delta$-complexes, and many localizations of spaces and simplicial sets are cataloged in \cite{white-localization}.

\begin{problem}
Left-induce a model structure on coalgebras over the CDC cooperad \cite{walter} (5.2), and then use the techniques of Section \ref{sec:preservation-coalgebra} to determine which left Bousfield localizations preserve this structure. 
\end{problem}

There are many other interesting examples of comonads on pointed spaces, detailed in \cite{perrone}. For example, the universal covering endofunctor (on path connected, locally path connected, and semi-locally simply connected spaces) can be viewed as a comonad, whose coalgebras are connected spaces \cite{perrone} (5.3). The stream comonad begins with a topological monoid $N$ (thought of as time, like $\mathbb{R}$), then sends a space $X$ to the space of maps $N\to X$. Coalgebras are dynamical systems \cite{perrone} (5.4). Lastly, if $k$ is a field (with the discrete topology), $\cat M$ is the category of topological $k$-vector spaces, and $C$ is a nicely behaved coalgebra, then the category of $C$-comodules (itself a category of coalgebras over a comonad) admits a version of Hochschild (co)homology \cite{farinati-solotar} (Def. 3.8). In this context, is possible to lift localizations to the level of coalgebras \cite{farinati-solotar} (Sec. 5), and Hochschild cohomology commutes with localization \cite{farinati-solotar} (Sec. 6).

\begin{problem}
Produce left-induced model structures for the categories of coalgebras above, then use Theorem \ref{comonad.preservation.eq} to determine which left Bousfield localizations of spaces preserve these coalgebraic structures. Next, use the lifting results of \cite{farinati-solotar} (Sec. 5) to verify the conditions of Theorem \ref{comonad.preservation.eq} where $\cat M$ is the category of topological vector spaces. Next, extend the results of \cite{farinati-solotar} to topological Hochschild homology.
\end{problem}

Lastly, left-induced model structures on comodules in the category of pointed simplicial sets provide a natural setting for the study of Waldhausen $K$-theory \cite{wald}. If $X$ is a simplicial set and $E_*$ is a generalized reduced homology theory, then there are model structures on the category of comodules over $X_+$ and on the category of retractive spaces $R_X$, left-induced from the Kan model structure or from the $E_*$-local model structure on simplicial sets \cite{wald} (Sec. 4.1). Hence, the condition of Theorem \ref{clockwise.ccw} is met and $E_*$-localization lifts to comodules. Furthermore, this model structure on local comodules may also be right induced \cite{wald} (Sec. 4.2) and hence the equivalent conditions of Theorem \ref{comonad.preservation.eq} are met. The interplay between left Bousfield localization (with respect to twisted homology) and comodule structure is used to provide a comodule model for Waldhausen $K$-theory \cite{wald} (4.3) for spaces that fail to be simply connected.

In addition to coalgebras over cooperads in spaces, we can also study $R$-coalgebras, for a fixed commutative ring $R$, and we do so below.

\subsection{Simplicial model categories}

Every adjunction $
\nicexy@C+.5cm{\cat N \ar@<2pt>[r]^-{L} & \cat M \ar@<2pt>[l]^-{R}}$ gives rise to a comonad $K = LR$ on $\cat M$. There is a forgetful functor $U_K$ from $K$-coalgebras to $\cat M$ and this functor has a right adjoint. There is also a canonical coalgebra functor $Can_K: \cat N \to \coalgk$, defined by $Can_K(X) = L(X)$, described in \cite{hk} (Def. A.3), and it too has a right adjoint. In simplicial settings, cylinder objects can be constructed that allow for left-induced model structures in great generality. Specifically, if $\cat N$ and $\cat M$ are simplicial model categories, if the adjunction is a simplicial Quillen adjunction, and if the forgetful functor $U: \coalgk \to \cat M$ takes every $K$-coalgebra to a cofibrant object of $\cat M$, then $\coalgk$ has a model structure left-induced from $\cat M$ \cite{hk} (Thm. 3.2).

The authors of \cite{hk} consider several examples of this general framework, and we reproduce two of those examples here. First, we consider the free abelian comonad on simplicial abelian groups \cite{hk} (Cor. 3.3), where the model structure on $K$-coalgebras is left induced from the projective model structure on simplicial abelian groups. The model structure on $K$-coalgebras may also be viewed as left-induced from $sSet_*$ along the canonical $K$-coalgebra functor \cite{hk} (Thm. 3.4).
We can use Theorem \ref{comonad.preservation.eq} to obtain preservation results for these coalgebras under the left Bousfield localizations of simplicial sets and of simplicial abelian groups cataloged in \cite{white-localization}. There is also a one-reduced model structure on $K$-coalgebras \cite{hk} (Lemma 3.5), left induced from the model structure on $K$-coalgebras (and hence from simplicial abelian groups, and from simplicial sets), so our results can be used to lift localizations to this setting.

Another example of a simplicial Quillen adjunction satisfying the conditions of \cite{hk} (Thm. 3.2) are suspension and loops adjunctions $(\Sigma^r,\Omega^r)$ for $1\leq r \leq \infty$. Letting $K$ denote the associated comonad of such an adjunction, there are model structures on $K$-coalgebras left-induced along the forgetful functor, from $sSet_*$ in case $r < \infty$ \cite{hk} (Cor. 3.8) and from the positive stable model structure on symmetric spectra in case $r = \infty$ \cite{hk} (Cor. 3.9). Hence, we can apply Theorem \ref{comonad.preservation.eq} in these contexts, to lift left Bousfield localizations of spaces and spectra (cataloged in \cite{white-localization}) to these categories of coalgebras.

\subsection{Simplicial coalgebras} \label{subsec:sCoalg}

Let $R$ be a commutative, unital ring. The category of simplicial $R$-coalgebras has a model structure left induced from the projective model structure on simplicial $R$-modules \cite{raptis} (Thm. A). Furthermore, if $R$ is a field, then the category of coassociative, counital simplicial $R$-coalgebras has a similar left-induced model structure  \cite{sore} (Thm. 3.5), as does the category of cocommutative simplicial $R$-coalgebras \cite{sore} (3.2) first discovered by Goerss.

It is natural to ask whether the Dold-Kan equivalence induces a Quillen equivalence between this model structure on simplicial coalgebras, and the model structure on dg coalgebras of Remark \ref{remark:bounded}. Interestingly, as shown in \cite{sore} (Prop. 4.16), it does not. This stands in stark contrast to the positive results for algebras over operads, obtained in \cite{white-yau3}.

Nevertheless, there is a model structure on simplicial $R$-coalgebras (where $R$ is any commutative, unital ring), left-induced along the Dold-Kan adjunction from the model structure on dg coalgebras of Remark \ref{remark:bounded}:
\begin{equation}\label{eqn:dold-kan}
\nicexy@C+.5cm{\coalg_R(\sMod_R) \ar@<2pt>[r]^-{N} & \coalg_R(\Chr^{\geq 0}) \ar@<2pt>[l]^-{\Gamma_{\coalg}}}
\end{equation}
\noindent by \cite{peroux22} (4.2). With these model structures on $R$-coalgebras, we can apply Theorem \ref{comonad.preservation.eq} to study preservation of coalgebras under the left Bousfield localizations of spaces and of simplicial $R$-algebras considered in \cite{white-localization}.

Furthermore, the category of reduced simplicial sets (i.e., simplicial sets with a single vertex) admits many model structures \cite{rivera} (Thm. 1) whose homotopy theories can be modeled using categories of connected simplicial cocommutative coalgebras over a field $k$ \cite{rivera} (Thm. 2). These model structures are different from those obtained via derived Koszul duality \cite{rivera} (Remark 3.5.5). The first of these model structures on coalgebras is left induced along the following left adjoint functor \cite{rivera} (Thm. 7.3.1(1)) 
$$\Omega = \Cobar \circ N_*: \sCoCoalg_k^0 \to \dgAlg_k.$$
The second is a left Bousfield localization of the first, and the third is a further left Bousfield localization \cite{rivera} (Thm. 7.3.1). Hence, we are in the situation of Theorem \ref{coalgklcm-admissible}, and Theorem \ref{comonad.preservation.eq} can be used to determine when these coalgebraic structures are preserved by the left Bousfield localizations of \cite{rivera}.

\subsection{Simplicial presheaves}

Let $\cat C$ be a small Grothendieck site, and let $\cat M$ denote the category of simplicial presheaves $\sPre(\cat C)$. Let $\cat R$ be a presheaf of commutative unital rings on $\cat C$. Then, there is a model structure on the category of simplicial $\cat R$-coalgebras left induced from Jardine's local model structure on simplicial $\cat R$-algebras \cite{raptis} (Thm. A). That is, Raptis considers coalgebras left induced from a local model structure.

Theorem \ref{coalgklcm-admissible} gives an alternative approach. We could instead pass to coalgebras first, and then localize. In particular, this technique yields a homotopy sheafification functor on the category of coalgebras.

Additionally, there are many interesting left Bousfield localizations of simplicial presheaves that one can now study using Theorem \ref{comonad.preservation.eq}, e.g., the localizations discussed in \cite{chorny-white} related to functor calculus, or the localizations of \cite{Reedy-paper, bous-loc-semi} related to the homotopy theory of locally constant presheaves, with applications to the homotopy theory of $n$-operads \cite{bdw1, bdw2} and to the Baez-Dolan stabilization hypothesis \cite{batanin-white-baez-dolan, white-oberwolfach}.

Lastly, when $k$ is a perfect field, Raptis studies the left Bousfield localization of $\M$ with respect to the class of $\mathscr{F}$-homology equivalences, where $\mathscr{F}$ is the constant presheaf at $k$. Raptis proves that this localization respects the coalgebra structure, but his main result \cite{raptis} (Thm. B) and its extension to the equivariant setting where a Galois group acts \cite{raptis} (Thm. 6.5), are stated at the level of homotopy categories. The techniques of the present paper should make it possible to get stronger statements, on the model category level, and to more deeply study the relationship between left Bousfield localizations of simplicial presheaves, and the model structure on $\cat R$-coalgebras \cite{raptis} (Thm. A).

\subsection{Model structure on algebraically cofibrant objects}

If $\M$ is a combinatorial and simplicial model category, then there is a comonad $c$ whose coalgebras are the algebraically cofibrant objects of $\M$. The forgetful functor $U:\M_c\to \M$ left-induces a model structure on the category $\M_c$ of $c$-coalgebras that is Quillen equivalent to $\M$ \cite{ching-riehl} (Theorem 2.4). This provides a valuable way to replace a given model category with a Quillen equivalent one in which all objects are cofibrant, that the first author used in \cite{chorny-white}. If $\cat C$ is a set of morphisms in $\M$, then $\lcm$ is again combinatorial and simplicial, hence the category of coalgebras $(\lcm)_c$ has a left-induced model structure by \cite{ching-riehl} (Theorem 2.4). Thus, we are in the setting of Theorem \ref{clockwise.ccw}. Because $\M_c$ is Quillen equivalent to $\M$, they have the same homotopy category (and similarly for their left Bousfield localizations), so Definition \ref{crt-comonad} is automatic, and hence any left Bousfield localization satisfies the conditions of Theorem \ref{comonad.preservation.eq} for this comonad. Hence, using \cite{ching-riehl} (Theorem 2.4) to replace a model category by a Quillen equivalent one with all objects cofibrant respects left Bousfield localization. This was an important part of the theory of \cite{chorny-white}, which applied left Bousfield localization to a model category of functors between two model categories, to produce a model structure where the fibrant objects are homotopy functors (i.e., preserve weak equivalences). Further localizations of this model structure provide a natural setting for Goodwillie calculus \cite{chorny-white}.

\end{document}